\documentclass[12pt]{amsart}

\usepackage{sfilip_presets}
\usepackage{amsmath}
\usepackage{amsthm}

\newif \iffig
\figtrue

\iffig
  \usepackage{wrapfig,graphicx}
  
  \numberwithin{figure}{subsection}
\fi

\newenvironment{wrapped}[1]
{\def\wrappedcurrent{#1}%
	\setlength{\columnwidth}{\parshapelength\numexpr\prevgraf+2\relax}%
	\csname #1\endcsname}
{\csname end\wrappedcurrent\endcsname}

\usepackage[stretch=10, final]{microtype}
\usepackage{lmodern}
\usepackage[T1]{fontenc}
\usepackage{etoolbox}
\usepackage{needspace}

\AtBeginEnvironment{theorem}{\Needspace{5\baselineskip}}

\usepackage{titletoc}

\let\oldtocsection=\tocsection
\let\oldtocsubsection=\tocsubsection
\let\oldtocsubsubsection=\tocsubsubsection

\renewcommand{\tocsection}[2]{\hspace{0em}\oldtocsection{#1}{#2}}
\renewcommand{\tocsubsection}[2]{\hspace{1.75em}\oldtocsubsection{#1}{#2}}
\renewcommand{\tocsubsubsection}[2]{\hspace{2em}\oldtocsubsubsection{#1}{#2}}

\usepackage{etoolbox}
\makeatletter
\patchcmd{\@maketitle}
  {\ifx\@empty\@dedicatory}
  {\ifx\@empty\@date \else {\vskip3ex \centering\footnotesize\@date\par\vskip1ex}\fi
   \ifx\@empty\@dedicatory}
  {}{}
\patchcmd{\@adminfootnotes}
  {\ifx\@empty\@date\else \@footnotetext{\@setdate}\fi}
  {}{}{}
\makeatother

\renewcommand\labelenumi{(\roman{enumi})}
\renewcommand\theenumi\labelenumi

\makeatletter
\renewcommand{\@secnumfont}{\bfseries}
\makeatother
\makeatletter
\renewcommand{\section}{
\@startsection {section}
				   {1}
				   {\z@}%
                                   {-3.5ex \@plus -1ex \@minus -.2ex}%
                                   {2.3ex \@plus.2ex}%
                                   {\centering \normalfont \Large \bfseries \sffamily}}%
\makeatother

\makeatletter
\renewcommand{\subsection}
{\@startsection{subsection}
				     {2}
				     {\z@}%
                                     {-3.25ex\@plus -1ex \@minus -.2ex}%
                                     {1.5ex \@plus .2ex}%
                                     {\normalfont \large \bfseries \sffamily}}
\makeatother

\makeatletter
\renewcommand{\subsubsection}
{\@startsection{subsubsection}
				     {3}
				     {\z@}%
                                     {3.25ex \@plus1ex \@minus.2ex}%
				     {-\fontdimen 2\font }
                                     {\normalfont\normalsize\bfseries}}
\makeatother

\makeatletter
\renewcommand{\paragraph}{%
\@startsection {paragraph}
				    {4}
				    {\z@}
				    {\z@}
				    {-\fontdimen 2\font }
				    {\normalfont\normalsize\bfseries}
}
\makeatother

\PassOptionsToPackage{hyphens}{url}\usepackage{hyperref}

\makeatletter
\g@addto@macro{\UrlBreaks}{\UrlOrds}
\makeatother
\usepackage{color}
\definecolor{darkred}{rgb}{0.4,0,0}
\definecolor{darkgreen}{rgb}{0,0.5,0}
\definecolor{darkblue}{rgb}{0,0,0.4}
\hypersetup{
    pdftitle={Notes on the Multiplicative Ergodic Theorem},	
    pdfauthor={Simion Filip},	
    pdfsubject={Ergodic Theory, Dynamical Systems},		
    pdfkeywords={},	
    pdfnewwindow=true,		
    colorlinks=true,		
    linkcolor=darkblue,		
    citecolor=darkred,		
    filecolor=darkblue,		
    urlcolor=darkblue,		
    pdfborder={0 0 0},
    breaklinks=true
}

\def\subsek~{\S{}}

\def\equationautorefname~#1\null{%
  Eqn.~(#1)\null
}

\newtheoremstyle{mytheoremstyle}
    {5pt}	                
    {5pt}                    	
    {\itshape}                  
    {}                          
    {\bfseries}                 
    {.}                         
    {.5em}                      
    {}  			

\newtheoremstyle{mydefinitionstyle}
    {5pt}	                
    {5pt}                    	
    {}                  	
    {}                          
    {\bfseries}                 
    {.}                         
    {.5em}                      
    {}  			

\swapnumbers			
\theoremstyle{mytheoremstyle}			

\newtheorem{theorem}{Theorem}[subsection]	

\numberwithin{equation}{subsection}		

\makeatletter					
\let\c@theorem \c@equation
\makeatother

\makeatletter					
    \let\c@subsubsection\c@equation
\makeatother

\makeatletter				
  \let\c@figure\c@equation		
\makeatother
\usepackage{aliascnt}

\newaliascnt{lemma}{theorem}  
\newtheorem{lemma}[lemma]{Lemma}  
\aliascntresetthe{lemma}

\newaliascnt{proposition}{theorem}
\newtheorem{proposition}[proposition]{Proposition}
\aliascntresetthe{proposition}

\newaliascnt{corollary}{theorem}  
\newtheorem{corollary}[corollary]{Corollary}  
\aliascntresetthe{corollary}

\theoremstyle{mydefinitionstyle}		

\newaliascnt{exercise}{theorem}  
\newtheorem{exercise}[exercise]{Exercise}  
\aliascntresetthe{exercise}

\newaliascnt{definition}{theorem}  
\newtheorem{definition}[definition]{Definition}  
\aliascntresetthe{definition}

\newaliascnt{remark}{theorem}  
\newtheorem{remark}[remark]{Remark}  
\aliascntresetthe{remark}

\newaliascnt{example}{theorem}  
\newtheorem{example}[example]{Example}  
\aliascntresetthe{example}

\newaliascnt{question}{theorem}  
  
\aliascntresetthe{question}

\newif\ifdebug
\debugfalse

\ifdebug
  \usepackage[status=draft,author=]{fixme}
  \fxsetup{inline, marginclue,theme=color}
\else
  \usepackage[status=final, author=]{fixme}
\fi

\newenvironment{custom_description}[1]%
  {\begin{list}{}%
   {\renewcommand\makelabel[1]{##1:\hfill}%
   \settowidth\labelwidth{\makelabel{#1}}%
   \setlength\leftmargin{\labelwidth}%
   \addtolength\leftmargin{\labelsep}}}%
  {\end{list}}

\renewcommand{\labelenumi}{\theenumi}%
\renewcommand{\theenumi}{(\roman{enumi})}%

\begin{document}
%
\title[Notes on MET]{Notes on the Multiplicative Ergodic Theorem}
%
\thanks{Revised \textsc{\today} }

\date{			}

\author{
Simion Filip
}
\address{
\parbox{0.5\textwidth}{
Department of Mathematics\\
Harvard University\\
Cambridge, MA 02139\\}
	}
\email{{sfilip@math.harvard.edu}}
%
%
%
\begin{abstract}
The Oseledets Multiplicative Ergodic theorem is a basic result with numerous applications throughout dynamical systems.
These notes provide an introduction to this theorem, as well as subsequent generalizations.
They are based on lectures at summer schools in Brazil, France, and Russia.
\end{abstract}

%
%
\maketitle
%
\noindent\hrulefill
\tableofcontents
\noindent \hrulefill

\ifdebug
  \listoffixmes
\fi

\section{Introduction}

The Oseledets multiplicative ergodic theorem is a basic result with applications throughout dynamical systems.
It was first proved by Oseledets \cite{Oseledets}, with previous work on random multiplication of matrices by Furstenberg and Furstenberg--Kesten \cite{Furstenberg_Kesten}.

In its basic form, it describes the asymptotic behavior of a product of matrices, sampled from a dynamical system.
For instance, start with two fixed $d\times d$ matrices $A$ and $B$ and flip a coin with two sides $ n $ times to obtain a random product of matrices of the form $C_n:=ABAABA\ldots $ writing $A$ for heads and $B$ for tails (new matrices are added on the left of the sequence).
The usual law of large numbers describes the frequency of $A$'s and $B$'s, while its noncommutative version describes the resulting product matrix.

The description is in terms of a collection of numbers $\lambda_1\geq \cdots \geq \lambda_d$ called Lyapunov exponents, and a collection of (random) nested subspaces $V^{\leq \lambda_i}\subset \bR^d$ called Lyapunov subspaces.
Applying the (random) matrix $C_n$ to a vector $v\in V^{\leq\lambda_i}\setminus V^{\leq \lambda_{i+1}}$ expands it by a factor of about $e^{\lambda_i \cdot n}$.

\paragraph{Smooth Dynamics}
One major motivation for a matrix version of the ergodic theorem was smooth dynamics.
Starting from a diffeomorphism of a manifold, its induced action on the tangent bundle can be viewed as a collection of linear maps at each point of the manifold.
Iterating the diffeomorphism, one studies the successive multiplication of the corresponding matrices along the orbit.

Pesin theory \cite{Pesin_families,Pesin_characteristic} describes the local behavior of the iterates of the diffeomorphism in terms of the Lyapunov exponent and subspaces.
The dimension of an invariant measure, the Lyapunov exponents, and the entropy of the diffeomorphism are related by the Ledrappier--Young formula \cite{LY1,LY2}.

\paragraph{Extensions}
Ruelle \cite{Ruelle} extended the setting of the Oseledets theorem to infinite dimensional Hilbert spaces.
A geometric interpretation of the original theorem was later given by Kaimanovich \cite{Kaimanovich}, viewing a matrix as an isometry of a symmetric space.
The statement of the Oseledets theorem became equivalent to one about divergence of points along a geodesic in the symmetric space.

The work of Karlsson--Margulis \cite{Karlsson_Margulis} extended the scope of the theorem to isometries of non-positively curved metric spaces.
The statement is again in terms of divergence of points along a geodesic.

A subsequent extension by Karlsson--Le\-drap\-pier \cite{Karlsson_Ledrappier} included \emph{all} proper metric spaces.
Divergence along a geodesic was replaced by linear divergence of the values of a horofunction.
Finally, Gou\"ezel--Karlsson \cite{Gouezel_Karlsson} extended the last result to include semi-contractions, and not just isometries of a metric space.

\paragraph{Applications to Rigidity}
One step in a proof of Margulis superrigidity \cite{Margulis_arithmeticity} uses the Oseledets theorem.
It views the collection Lyapunov subspaces as a point on a flag manifold, which is the boundary of a symmetric space.
The Oseledets theorem then gives a boundary map, which is a starting point for many proofs of rigidity statements.
This point of view is detailed in Zimmer's work \cite{Zimmer}.
Monod's approach \cite{Monod} to rigidity in non-positive curvature is reversed in \autoref{sec:mean_MET} to give a mean non-commutative ergodic theorem.

\paragraph{Textbook treatments}
A transparent discussion of the Oseledets theorem with many applications is in the lecture notes of Ledrappier \cite{Ledrappier_lectures}.
Some monographs which deal with this circle of questions are the ones by Katok--Hasselblatt \cite{KH}, Ma\~{n}\'{e} \cite{Mane}, Viana \cite{Viana}, and Zimmer \cite{Zimmer}.
A treatment of the Oseledets theorem similar to the one in the present text appears in the notes of Bochi \cite{Bochi_notes}.
The notes of Karlsson \cite{Karlsson_notes} are another useful source, with many ideas there influencing the presentation here.

\paragraph{Outline}
\autoref{sec:OseledetsMET} discusses the classical form of the Oseledets theorem.
The presented proof is based on some geometric arguments and avoids the Kingman subadditive ergodic theorem.

\autoref{sec:geometric_Oseledets} contains a more geometric point of view on the Oseledets theorem, due to Kaimanovich \cite{Kaimanovich}.
It is formulated in terms of sequences of points tracking geodesics in symmetric spaces associated to Lie groups.

\autoref{sec:general_ncet} contains a very general non-commutative ergodic theorem that applies to isometries of general proper metric spaces.
The linear divergence to infinity is now detected by a horofunction, rather than linear tracking of a geodesic.

\autoref{sec:mean_MET} discusses a mean version of non-commutative ergodic theorems in non-positively curved metric spaces.
The general notion of ``direct integral'' of metric spaces allows for a rather elementary treatment.

\section{Oseledets Multiplicative Ergodic Theorem}
\label{sec:OseledetsMET}

\subsection{Basic ergodic theory}

For a thorough introduction to ergodic theory, one can consult the monographs of Katok--Hasselblatt \cite{KH}, Einsiedler--Ward \cite{Einsiedler_Ward}, or Walters \cite{Walters}.

\subsubsection{Notation}
Let $\Omega$ be a separable, second-countable metric space, $\cB$ its Borel $\sigma$-algebra and $\mu$ a probability measure on $\Omega$.
Let $T:\Omega\to \Omega$ be a measurable transformation, i.e. if $A$ is measurable then $T^{-1}(A)$ is also measurable.
Define the push-forward measure $T_*\mu$ by
\begin{align*}
T_*\mu(A) := \mu\left(T^{-1}(A)\right) \textrm{ where }A\in \cB.
\end{align*}
The measure $\mu$ is $T$-invariant if $T_*\mu=\mu$ and in this case $(\Omega,\cB,\mu,T)$ is called a probability measure-preserving system.
This will also be denoted below as $T\curvearrowright (\Omega,\mu)$.

\subsubsection{Ergodicity}
By definition $(\Omega,\cB,\mu,T)$ is ergodic if the only measurable $T$-invariant sets have either full or null measure.
In other words, if $T^{-1}(A)=A$ then $\mu(A)=0$ or $1$.
Equivalently, if $\Omega=A\coprod B$ is a $T$-invariant decomposition into measurable sets, then $\mu(A)=1$ or $\mu(B)=1$.
Using the ergodic decomposition (\cite[Ch. 4.2]{Einsiedler_Ward}) most questions can be reduced to ergodic systems.

A starting point for many results is the Birkhoff, or Pointwise Ergodic Theorem.
\begin{theorem}[Birkhoff Pointwise Ergodic Theorem]
  \label{thm:Birkhoff}
 Assume $T\curvearrowright(\Omega,\mu)$ is an ergodic probability measure-preserving system and let $f\in L^1(\Omega,\mu)$ be an integrable function.
 Then for $\mu$-a.e. $\omega\in \Omega$ we have
 \begin{align*}
 \lim_{N\to \infty} \frac 1N \sum_{i=0}^{N-1} f(T^i\omega) \to \int_{\Omega} fd\mu.
 \end{align*}
\end{theorem}
The quantity on the right is the ``space average'' of $f$, while the quantity on the left is the ``time average''.

\begin{remark}
\label{rmk:minus_infty}
 One can slightly relax the $L^1$-integrability assumption above.
 Let $f=f^+ + f^-$ with $f^+\geq 0$ and $f^-\leq 0$.
 If $f^+\in L^1(\Omega,\mu)$, then the same conclusion in \autoref{thm:Birkhoff} holds, except that $-\infty$ is allowed as a limit.
\end{remark}

\subsubsection{Variant}
Instead of additive averages, one can also consider multiplicative ones.
With the notation as in the theorem, set $g:=\exp(f)$ (equivalently, $f=\log g$).
Then for a.e. $\omega\in \Omega$ we have
\[
\lim_{N\to \infty} \left( g(\omega)g(T\omega)\cdots g(T^{N-1}\omega)\right)^{1/N} \to \exp\left(\int_{\Omega}\log g \, d\mu\right)
\]

\subsection{The Oseledets theorem}

\begin{example}
Suppose that $M$ is a smooth manifold and $F:M\to M$ is a smooth diffeomorphism preserving a measure $\mu$ (for instance, a volume form).
The diffeomorphism induces a map on the tangent bundle of $M$; if $p\in M$ is a point then
\begin{align*}
D_p F: T_p M \to T_{F(p)}M \textrm{ is a linear map.}
\end{align*}
In analogy with the Birkhoff ergodic theorem, one can inquire about the asymptotic behavior of the $N$-fold composition
\begin{align*}
D_{F^{N-1}(p)}F \circ \cdots \circ D_p F : T_p M \to T_{F^{N}(p)}M
\end{align*}
An answer is given by the Oseledets Multiplicative Ergodic theorem.
\end{example}

\subsubsection{Vector bundles}
For $(\Omega,\cB,\mu)$ consider vector bundles $V\to \Omega$.
By definition, this is a collection of sets $U_\alpha$ which cover $\Omega$, as well as gluing maps $\phi_{\alpha,\beta}:U_{\alpha}\cap U_{\beta} \to \GL_n\bR$ satisfying the compatibility condition $\phi_{\gamma,\alpha}\cdot \phi_{\beta,\gamma}\cdot \phi_{\alpha,\beta} = \id$.
The vector bundle $V$ is then defined by gluing the pieces $U_\alpha\times \bR^n$ using the identifications:
\begin{align*}
V := \left( \coprod_{\alpha} U_\alpha \times \bR^n\right) / \raisebox{-.2em}{$(\omega,v)\sim (\omega,\phi_{\alpha,\beta}v) \textrm{ for } \omega \in U_\alpha\cap U_\beta$}
\end{align*}
The sets $U_\alpha$ can be measurable or open.
Similarly, the gluing maps can be measurable, continuous, smooth, depending on the qualities of $\Omega$.

Any vector bundle can be measurably trivialized, i.e. it is measurably isomorphic to $\Omega\times\bR^n$.
However, intrinsic notation will be more useful in the sequel.
For $\omega\in \Omega$ using it as subscript, e.g. $V_\omega$, will denote the fiber of $V$ over $\omega$.

\subsubsection{Cocycles.}
Let $T:\Omega\to \Omega$ be a probability measure preserving transformation, and $V\to \Omega$ a vector bundle.
Then $V$ is a \emph{cocycle} over $T$ if the action of $T$ lifts to $V$ by linear transformations.
In other words, there are linear maps
\[
T_\omega:V_\omega \to V_{T\omega}
\]
and these maps vary measurably, continuously, or smoothly with $\omega$.
For simplicity, assume that $T_\omega$ is always \emph{invertible}.

\subsubsection{Oseledets Multiplicative Ergodic Theorem.}
To motivate the next result, consider a single matrix $A$ acting on $\bR^n$, say with positive eigenvalues $e^{\lambda_1}>e^{\lambda_2}>\cdots > e^{\lambda_k}$ (perhaps with multiplicities).
For a vector $v\in \bR^n$, consider the behavior of $\norm{A^Nv}$ as $N$ gets large.
For typical $v$, this will grow at rate $e^{N\cdot \lambda_1}$.
But for $v$ contained in the span of eigenvectors with eigenvalues $e^{\lambda_2}$ or less, the growth rate of $\norm{A^Nv}$ will be different.
Repeating the analysis gives a filtration of $\bR^n$ by subspaces with different order of growth under iteration of $A$.

\begin{remark}
 Throughout these notes, a metric on a vector space or vector bundle will mean a symmetric, possitive-definite bilinear form.
 In other words, it is a positive-definite inner product.
\end{remark}

The next result is a generalization of the Birkhoff theorem to cocycles over general vector bundles.
\begin{theorem}[Oseledets Multiplicative Ergodic Theorem]
  \label{thm:Oseledets}
 Suppose $V\to (\Omega,\cB,\mu,T)$ is a cocycle over an ergodic probability measure-preserving system.
 Assume that $V$ is equipped with a metric $\norm{-}$ on each fiber such that
 \begin{align}
   \label{eqn:cond_L1_bdd_cocycle}
   \int_{\Omega} \log^+\norm{T_\omega}_{op} d\mu(\omega) < \infty
 \end{align}
 Here $\log^+(x):=\max(0,\log x)$ and $\norm{-}_{op}$ denotes the operator norm of a linear map between normed vector spaces.
 
 Then there exist real numbers $\lambda_1>\lambda_2>\cdots >\lambda_k$ (with perhaps $ \lambda_k = -\infty $) and $T$-invariant subbundles of $V$ defined for a.e. $\omega\in \Omega$:
 \[
 0\subsetneq V^{\leq \lambda_k} \subsetneq \cdots \subsetneq V^{\leq \lambda_1} = V
 \]
 such that for vectors $v\in V^{\leq \lambda_i}_\omega\setminus V^{\leq \lambda_{i+1}}_\omega$ we have
 \begin{align}
 \label{eqn:growthofv}
   \lim_{N\to \infty} \frac 1N \log \norm{T^N v} \to \lambda_i
 \end{align}
\end{theorem}

\begin{remark}
 \leavevmode
 \begin{enumerate}
  \item The invariance of the subbundles $V^{\leq \lambda_i}$ means that $T_\omega:V_{\omega}\to V_{T\omega}$ takes $V^{\leq \lambda_i}_\omega \to V^{\leq \lambda_i}_{T\omega}$.
  The filtration $V^{\leq \lambda_\bullet}$ will be called the \emph{forward Oseledets filtration}.
  The numbers $\{\lambda_i\}$ are called \emph{Lyapunov exponents}.
  \item The multiplicity of an exponent $\lambda_i$ is defined to be $\dim V^{\leq \lambda_i}-\dim V^{\leq \lambda_{i+1}}$.
  Later it will be convenient to list exponents repeating them with their appropriate multiplicity.
  \item The Oseledets filtration and exponents are \emph{canonical}, since they are defined by the property in \autoref{eqn:growthofv}.
  If we have an exact sequence of bundles $ V\to W \to W/V $ then the exponents of $ W $ are the union of those in $ V $ and $ W/V $.
  To see this, one can apply successively \autoref{lemma:splitting} below to the filtration on $ W $ coming from the Oseledets filtrations on $ V $ and $ W/V $ (a similar construction will appear in the proof of the Oseledets theorem below).
 \end{enumerate}
\end{remark}

\subsubsection{Variant.}
Suppose now that $T:\Omega\to \Omega$ is an \emph{invertible} map, and the cocycle on $V$ for the map $T^{-1}$ satisfies the same assumptions as in \hyperref[thm:Oseledets]{the Oseledets Theorem \ref*{thm:Oseledets}}.
Applying the result to the inverse operator gives a set of $ k' $ exponents $\eta_j$ and the \emph{backwards} Oseledets filtration $V^{\leq \eta_j}$.
By construction if $v\in V^{\leq \eta_j}\setminus V^{\leq \eta_{j+1}}$ then
  \begin{align*}
  \lim_{N\to \infty} \frac 1 N \log \norm{T^{-N} v} \to \eta_j
  \end{align*}
  The only way for this to be compatible with the forward behavior of the vectors is if $\eta_j=-\lambda_{k+1-j}$ and $ k'=k $.
  Moreover, defining $V^{\lambda_j}:=V^{\leq \lambda_j}\cap V^{\leq \eta_{k+1-j}}$ gives a $T$-invariant direct sum decomposition
  \begin{align*}
   V = V^{\lambda_1} \oplus \cdots \oplus V^{\lambda_k}
  \end{align*}
  The defining dynamical property of this decomposition is that
  \begin{align*}
   0\neq v \in V^{\lambda_i} \Leftrightarrow \lim_{N\to \pm \infty} \frac 1N \log \norm{T^N v} = \lambda_i
  \end{align*}
  Note that when $N$ goes to $-\infty$, the sign in the $\frac 1N$ factor changes.
  
\begin{remark}
	\label{rmk:line_bdl}
 Suppose that the vector bundle $V$ is $1$-dimensional.
 The Oseledets theorem is then equivalent to the Birkhoff theorem.
 To see this, define
 \[
 f(\omega):= \log \frac{\norm{T_\omega v}}{\norm{v}} \textrm{ which is independent of the choice of } 0\neq v\in V_\omega.
 \]
 The integrability condition on the cocycle from \autoref{eqn:cond_L1_bdd_cocycle} is equivalent to the integrability of $f^+:=\max(0,f)\in L^1(\Omega,\mu)$.
 Then the Birkhoff theorem gives
 \begin{align*}
 \frac 1 N \left( f(\omega)+\cdots + f(T^{N-1}\omega) \right) = \frac 1N \log \frac {\norm{T^N v}}{\norm{v}} \to \int_\Omega f d\mu
 \end{align*}
 It follows that 
 \begin{align*}
 \frac 1N \log \norm{T^Nv} \to \int_\Omega f d\mu = \lambda_1
 \end{align*}
 Conversely, given $f\in L^1(\Omega,\mu)$, define $T:\Omega\times \bR \to\Omega\times \bR$ by 
 \begin{align*}
 T(\omega,v)=\left(T\omega,\exp(f(\omega))\cdot v\right)
 \end{align*}
 and use the standard norm on $\bR$ to deduce the Birkhoff theorem.
\end{remark}

\subsection{Proof of the Oseledets theorem}

The proof of \hyperref[thm:Oseledets]{the Oseledets Theorem \ref*{thm:Oseledets}} in this section will involve two preliminary results, stated below as \autoref{lemma:dichotomy} and \autoref{lemma:splitting}.
These will be proved in separate sections below.
The setup and notation is from \autoref{thm:Oseledets} above.

\begin{lemma}[Reducibility or growth dichotomy]
  \label{lemma:dichotomy}
  At least one, but perhaps both, of the following possibilities occur.
  \begin{enumerate}
   \item There exists $\lambda\in\bR$ such that for a.e. $\omega\in \Omega$ and for all nonzero $v\in V_\omega$
   \begin{align*}
   \lim_{N\to \infty} \frac 1 N\log \norm{T^N v} \to \lambda
   \end{align*}
   and the limit is uniform over all $ v $ with $ \norm{v}=1 $, and fixed $ \omega $.
   \item There exists a nontrivial proper $T$-invariant subbundle $E\subsetneq V$ which is defined at $\mu$-a.e. $\omega$.
  \end{enumerate}
  In other words, either there is a nontrivial subbundle with $0<\dim E < \dim V$, or vectors in $V$ exhibit growth with just one Lyapunov exponent.
\end{lemma}

I am grateful to the referee for pointing out that ideas similar to the ones used in the proof of the above lemma also appear in the article of Walters \cite{Walters_proof}.

\begin{remark}
  Suppose $V$ is a cocycle over $\Omega$ such that \autoref{eqn:cond_L1_bdd_cocycle} holds.
  Suppose further that $E\subset V$ is an a.e. defined $T$-invariant subbundle.
  Then the same boundedness condition \eqref{eqn:cond_L1_bdd_cocycle} holds for the bundles $E$ and $V/E$, equipped with the natural norms.
\end{remark}

In typical situations, one expects generic vectors to grow at maximal possible speed.
The lemma below shows that if maximal growth occurs in a proper subbundle, then the corresponding piece must split off as a direct summand.
\begin{lemma}[Unusual growth implies splitting]
  \label{lemma:splitting}
  Consider a short exact sequence of cocycles over $\Omega$
  \[
  0 \to E \into V \xrightarrow{p} F \to 0
  \]
  Assume there exist $\lambda_E,\lambda_F\in \bR$ such that for a.e. $\omega\in \Omega$
  \begin{align}
  \begin{split}
    \forall e\in (E_\omega\setminus 0) \textrm{ we have } & \frac 1N \log \norm{T^N e} \to \lambda_E\\
    \forall f\in (F_\omega\setminus 0) \textrm{ we have } & \frac 1N \log \norm{T^N f} \to \lambda_F
  \end{split}
  \end{align}
  and the limits are uniform over $ \norm{e}=1, \norm{f}=1 $ and fixed $ \omega $.
  
  If $\lambda_E>\lambda_F$ then the sequence is split, i.e. there exists a linear map
  \begin{align}
  \sigma:F\to V \textrm{ such that }V= E\oplus \sigma(F) \textrm{ and } p\circ \sigma  = \id_F
  \end{align}
  and this decomposition of $ V $ is $T$-invariant.
  The exponent $ \lambda_F $ is allowed to be $ -\infty $.
\end{lemma}

I am grateful to the referee for pointing out that a similar idea appears in Ma{\~n}e's treatment of the Oseledets theorem \cite[Lemma 11.6]{Mane}.

\begin{remark}
\leavevmode
  \begin{enumerate}
   \item The splitting constructed in \autoref{lemma:splitting} will be \emph{tempered}, i.e.
   \[
   \lim_{N\to\infty} \frac 1N \log \frac{ \norm{\sigma(T^Nv)} }{\norm{T^Nv}} = 0
   \]
   The Lyapunov exponent of $F$ and $\sigma(F)$ will thus be the same.
   \item In general, one expects that the maximal Lyapunov exponent on a subbundle should be less than the one on the entire bundle.
   What \autoref{lemma:splitting} says is that if this is not the case, then the cocycle must be a direct sum.
  \end{enumerate}
\end{remark}

\begin{proof}[Proof of \autoref{thm:Oseledets}]
	Proceed by induction on the dimension of the cocycle $ V $.
	The base case of line bundles follows from the Birkhoff ergodic theorem, as explained in \autoref{rmk:line_bdl}.
	Suppose now that the statement holds for all cocycles of dimension at most $ n-1 $ and that $ V $ is $ n $-dimensional.
	\autoref{lemma:dichotomy} implies that either the statement holds on $ V $ with a single exponent $ \lambda_1 $, or there exists a proper $ T $-invariant subbundle $ E\subset V $.
	Among all such possible $ E $, pick one that has minimal codimension.
	Then applying \autoref{lemma:dichotomy} again to $ V/E $ it follows that all vectors in $ V/E $ grow at the same rate $ \lambda' $.
	
	Applying the inductive assumption to $ E $ gives exponents $ \lambda_1> \cdots >\lambda_k $ and subbundles $ E^{\leq \lambda_i} $ with the required properties.
	Note that $ E^{\leq \lambda_i} $ are also naturally subbundles in $ V $.
	If $ \lambda_1\leq \lambda' $ then we are done.
	Otherwise, apply \autoref{lemma:splitting} to the sequence
	\begin{align*}
	0\to E/E^{\leq \lambda_2} \into V/E^{\leq \lambda_2} \onto V/E \to 0
	\end{align*}
	to obtain a splitting map $ \sigma:V/E \to V/E^{\leq \lambda_2} $.
	Then the preimage of $ \sigma(V/E) $ under the projection $ V \to V/E^{\leq \lambda_2} $ gives us a subbundle denoted $ V^{\leq \lambda_2} $ with the following two properties.
	First, any vector outside $ V^{\leq \lambda_2} $ will grow at rate $ \lambda_1 $.
	Second, $ V^{\leq \lambda_2} $ maps naturally to $ V/E $ (which has exponent $ \lambda' $).
	
	Apply iteratively \autoref{lemma:splitting} to $ V^{\leq \lambda_i} $ (starting now with $ V^{\leq \lambda_2 }$) to construct the Oseledets filtration as above, repeating the process until $ \lambda_i \leq \lambda' $.
\end{proof}

\subsection{Proof of \autoref{lemma:dichotomy}}

Consider the bundle $V\to \Omega$ and the associated projective space bundle $\bP(V)\xrightarrow{\pi} \Omega$.
Since the transformation $T$ acts on $V$ by linear maps, its action extends to $\bP(V)$ by projective-linear transformations.

Define the space of probability measures on $\bP(V)$ which project to the measure $\mu$ on $\Omega$:
\begin{align}
\label{eqn:prob_msr_proj_bdl}
\cM^1(\bP(V),\mu) := \{\eta \textrm{ prob. measure on }\bP(V) \textrm{ with }\pi_*\eta=\mu \}
\end{align}
Since $\mu$ is $T$-invariant, the action of $T$ on $\bP(V)$ naturally extends to an action on $\cM^1(\bP(V),\mu)$.
A weak-* topology on $ \cM^1(\bP(V,\mu)) $ is described in \cite[\S~4.2.3]{Viana} using duality with measurable functions on the total space, which are continuous on $ \mu $-a.e. fiber.
A related discussion can be found in \cite[\S~4]{Bochi_notes}.

\begin{exercise}\leavevmode
\begin{enumerate}
 \item \textbf{Krylov--Bogoliubov.} Let $S:\Omega\to \Omega$ be a homeomorphism of a compact separable metric space.
  Prove that the space of probability measures on $\Omega$ is weak-* and sequentially compact, and that it has an $S$-invariant measure.
  \item \textbf{Krylov--Bogoliubov in families.}
  Prove that the space of probability measures $\cM^1(\bP(V),\mu)$ (see \eqref{eqn:prob_msr_proj_bdl}) is weak-* and sequentially compact, and that it has at least one $T$-invariant measure.
\end{enumerate}
\end{exercise}

Define now the function $f:\bP(V)\to \bR$ by
\begin{align*}
 f([v]) := \log\left( \frac{\norm{Tv}}{\norm{v}} \right)
\end{align*}
Integrating against $f(-)$ gives a continuous function on the space of measures:
\begin{align}
\begin{split}
 \int f :\cM^1(\bP(V),\mu) &\to \bR\\
 \eta &\mapsto \int_{\bP(V)}f {\rm d}\eta
\end{split}
\end{align}
The set of $T$-invariant measures, denoted $\cM^1(\bP(V),\mu)^T$ is non-empty and weak-* compact, so this function achieves a minimum on it.
Moreover, the set of measures on which the minimum is achieved is a closed convex subset.
Therefore, an extremal point of this convex set exists.
This is a measure $\eta\in \cM^1(\bP(V),\mu)$ which is also ergodic for the $T$-action on $\bP(V)$; ergodicity follows by the choice of $ \eta $ as an extreme point in the minimizing convex set.

Applying the Birkhoff ergodic theorem to the function $f$ and the measure $\eta$, it follows that for $\eta$-a.e. $[v]\in \bP(V)$
\begin{align}
\label{eqn:orbit_sum}
 \frac 1N \left( f([v]) + \cdots + f(T^{N-1}[v]) \right) = \frac 1 N \log \left(\frac{\norm{T^Nv}}{\norm{v}}\right) \to \int_{\bP(V)} fd\eta =:\lambda
\end{align}
In each fiber $\bP(V_\omega)\subset \bP(V)$ consider the set $M$ of vectors for which the above limit holds.
This is a $T$-invariant set, and for $\mu$-a.e. $\omega$ this set is non-empty.
Moreover, the fiberwise span of the $[v]\in M$ gives a $T$-invariant subbundle $E\subset V$.

If this is a proper subbundle, then we are in case (ii) of \autoref{lemma:dichotomy}.
Suppose therefore that $E=V$.

In this case, any vector in $V$ can be written as a linear combination of vectors with asymptotic norm growth rate $\lambda$.
Thus, their asymptotic growth rate is at most $\lambda$.
In fact, by choice of $ \lambda $ their asymptotic growth rate has to be exactly $\lambda$, uniformly in $ [v] $ in a fixed fiber.

To see this, suppose by contradiction that there exists a $ \mu $-generic $ \omega $ (i.e. the forward orbit of $ \omega $ weak-* approaches $ \mu $), an $ \epsilon>0 $ and a sequence $ [v_i] \in \bP(V_\omega) $ such that
\begin{align}
\label{eqn:slow_growth}
\limsup_{N_i\to \infty} \frac 1 {N_i} \log (\norm{T^{N_i} v_i}) \leq \lambda -\epsilon
\end{align}
where $ \norm{v_i}=1 $ and $ N_i $ is some sequence tending to $ +\infty $.

Consider now probability measures on $ \bP(V) $ given by normalized $ \delta $-masses on the orbits of length $ N_i $ of $ [v_i] $, call them $ \eta_i $.
Let $ \eta_\epsilon $ be one of their weak-* limits: it is $ T $-invariant by construction and projects to $ \mu $ since $ \omega $ is $ \mu $-generic.

However, from \autoref{eqn:slow_growth} we see that $ \limsup \int_{\bP(V)}f  {\rm d}\eta_i \leq \lambda -\epsilon $ (see \autoref{eqn:orbit_sum} for the relation between orbit sum and function evaluation).
It follows that also $ \int_{\bP(V)} f {\rm d}\eta_\epsilon \leq \lambda -\epsilon $ which contradicts the definition of $ \lambda $.
\hfill \qed

\begin{remark}
The proof of \autoref{lemma:dichotomy} shows that the filtration constructed will, in fact, give the correct sequence of Lyapunov exponents.
However, \autoref{lemma:splitting} might be of intrinsic interest.
\end{remark}

\subsection{Proof of \autoref{lemma:splitting}}
The setup is a short exact sequence of vector bundles
\begin{align*}
 0 \to E \to V \xrightarrow{p} F \to 0
\end{align*}
Pick any lift $\sigma_0: F \to V$ such that $V=E\oplus \sigma_0(F)$.
For instance, since $V$ has a metric, there is a natural identification $F\xrightarrow{\sim} E^\perp\subset V$.

The cocycle map $T$ now takes the form
\begin{align*}
 T = \begin{bmatrix}
      T_E & U \\
      0 & T_F
     \end{bmatrix}
\end{align*}
with the following linear maps
\begin{align}
\begin{split}
 T_{E,\omega}&: E_\omega \to E_{T\omega}\\
 T_{F,\omega}&: F_\omega \to F_{T\omega}\\
 U_\omega &: F_\omega \to E_{T\omega}
 \end{split}
\end{align}
Any other possible lift $\sigma:F \to V$ differs from $\sigma_0$ by a map $\tau : F\to E$.
Indeed, the difference $\tau=\sigma-\sigma_0$ is a map $F\to V$ which after composition back with the projection $p:V\to F$ is the zero map.
Therefore $\tau$ has image in $E$, and can be regarded as a map $\tau :F\to E$.

The condition that $\sigma = \sigma_0 + \tau$ is a splitting that diagonalizes $T$ is explicit.
Work in the decomposition $V = E \oplus \sigma_0(F)$.
Then a vector $\begin{bmatrix} e\\ f\end{bmatrix}$ is in $\sigma(F)$ if and only if $e=\tau(f)$.
But applying $T$ gives
\begin{align}
 T\begin{bmatrix}
   \tau(f)\\
   f
  \end{bmatrix}
  =
  \begin{bmatrix}
    T_{E,\omega} \circ \tau_\omega (f) + U_\omega(f)\\
    T_{F,\omega}(f)
  \end{bmatrix}
   \in V_{T\omega}
\end{align}
The condition that this vector is in $\sigma_{T\omega}(F)$ becomes
\begin{align*}
 T_{E,\omega} \circ \tau_{\omega}(f) + U_\omega(f) = \tau_{T\omega}\circ T_{F,\omega}(f).
\end{align*}
Note that this is an equality of maps from $F_\omega$ to $E_{T\omega}$.
The equation can be equivalently rewritten
\begin{align}
 \tau_\omega = T^{-1}_{E,\omega}\circ \tau_{T\omega} \circ T_{F,\omega} - T^{-1}_{E,\omega} \circ U_\omega
\end{align}
A formal solution of this equation is given by
\begin{align}
\label{eqn:formal_sln}
 \tau_\omega = -\sum_{n=0}^\infty (T^{n+1}_{E,\omega})^{-1} \circ U_{T^n \omega} \circ T^{n}_{F,\omega}
\end{align}
which uses the linear operators corresponding to iterating the maps $n$ times:
\begin{align}
\begin{split}
 T^{n}_{F,\omega} 	&: F_\omega\to F_{T^n \omega}		\\
 U_{T^n\omega} 		&: F_{T^n\omega} \to E_{T^{n+1}\omega}	\\
 (T^{n+1}_{E,\omega})^{-1} 	&: E_{T^{n+1}\omega} \to E_\omega 
 \end{split}
\end{align}
However, recall the assumption $\lambda_E>\lambda_F$, or equivalently $\lambda_F-\lambda_E<0$.
Each vector in $E$ and $F$ respectively grows exponentially at rate $\lambda_E$ and $\lambda_F$ respectively.

Therefore, the formal sum in \autoref{eqn:formal_sln} for a.e. $ \omega $ converges uniformly and gives the desired linear map.
For this, first note that the Birkhoff theorem gives for $\mu$-a.e. $\omega$
\begin{align}
\label{eqn:L1_basic_bound}
\norm{U_{T^n\omega} }_{op } = e^{o(n)}
\end{align}
since $ \log^+\norm{U}_{op} $ is in $ L^1 $.
Hence for a vector $ v\in F_\omega $ the expression $ \tau_\omega v $ is an infinite sum with each term bounded by an exponentially convergent series.
For this, use that
\begin{align*}
 \forall v_1\in E_{\omega}  \textrm{ we have } & \norm  {T^{n+1}_{E,\omega} v_1} =  e^{n \lambda_E + o(n)}\cdot \norm{v_1} 
 \intertext{and hence}
 \forall v_2\in E_{T^{n+1}\omega}  \textrm{ we have } & \norm { ({T^{n+1}_{E,\omega})}^{-1} v_2} =  e^{-n \lambda_E + o(n)}\cdot \norm{v_2}
\end{align*}
Hence for a.e. $ \omega $ and for a basis of $ F_\omega $ applying $ \tau_\omega $ is well-defined and converges uniformly.
\hfill \qed

\section{A geometric interpretation of the Oseledets theorem}
\label{sec:geometric_Oseledets}

The point of view on the Oseledets theorem developed in this section goes back at least to Kaimanovich \cite{Kaimanovich}.
The work of Karlsson and Margulis \cite{Karlsson_Margulis} provides a link between the point of view in this and the next sections.

\subsection{Basic constructions}

\subsubsection{Standard operations of linear algebra}
Starting with a collection of vector bundles, standard linear algebra operations produce new ones.
This applies in particular to cocycles over dynamical systems.

Suppose that $L$ and $N$ are two cocycles over the map $T:(\Omega,\mu) \to (\Omega,\mu)$, with Lyapunov exponents $\lambda_1\geq \cdots \lambda_l$ and $\eta_1\geq \cdots \geq \eta_n$.
The Lyapunov exponents are listed with multiplicities, and this will be the convention from now on.

Here is a list of constructions and corresponding Lyapunov exponents:
\begin{custom_description}{$\Hom(L,N)$}
 \item[$L\otimes N$] Exponents are $\{\lambda_i+\eta_j\}$ with $i=1\ldots l, j=1\ldots n$.
 \item[$L^\vee$ (dual)] Exponents are $\{-\lambda_i\}$ with $i=1\ldots l$.
 \item[$\Hom(L,N)$] Exponents are $\{\eta_j-\lambda_i\}$ with $i=1\ldots l, j=1\ldots n$.
 \item[$\Lambda^k(L)$] Exponents are $\{\lambda_{i_1}+\cdots + \lambda_{i_k}\}$ with $i_1<\cdots< i_k$.
\end{custom_description}
The Oseledets filtrations of the new cocycles can also be explicitly described in terms of those for the cocycles $L$ and $N$.

\subsubsection{Volume-preserving diffeomorphisms.}
Suppose that $F:M\to M$ is a diffeomorphism of a manifold which preserves a measure given by a volume form $\mu$.
Thus $\mu$ is a section of the top exterior power of the tangent bundle of $M$, denoted $\Lambda^n(TM)$.
Because $\mu$ is preserved by $T$ the Lyapunov exponent of $\Lambda^n(TM)$ is zero.
On the other hand, by the preceding discussion the exponent of $\Lambda^n(TM)$ is also the sum of the Lyapunov exponents on the tangent bundle $TM$.

\subsubsection{Symplectic cocycles.}
Suppose the rank $2g$ vector bundle $V\to \Omega$ carries a symplectic pairing denoted $\ip{-,-}$ which is preserved by the linear maps $T_\omega$.
Then the Lyapunov exponents of $V$ have the symmetry
\begin{align*}
 \lambda_1 \geq \cdots \geq \lambda_g \geq -\lambda_g \geq \cdots \geq -\lambda_1
\end{align*}
In fact, if $V^{\leq \lambda_i}$ denotes the forward Oseledets filtration, then the symplectic orthogonal of $V^{\leq \lambda_{i+1}}$ is $V^{\leq -\lambda_i}$.

For the symmetry of the Lyapunov spectrum, it suffices to note that the symplectic form gives an isomorphism of the cocycle $V$ and its dual cocycle $V^\vee$.
The exponents of $V^\vee$ are the negatives of that for $V$, so the claim follows.

The same construction gives the claim about filtrations.
The isomorphism given by the symplectic form respects the Oseledets filtrations on $V$ and $V^\vee$.
On the other hand, the Oseledets filtration on $V^\vee$ can be described in the general case as the dual, via annihilators, of the Oseledets filtration on $V$.

\begin{remark}
\leavevmode
\begin{enumerate}
 \item Consider a matrix in the symplectic group $A\in \Sp_{2g}(\bR)$.
 Suppose that $e^\lambda$ is an eigenvalue of $A$.
 Then $e^{-\lambda}$ is also an eigenvalue of $A$, and this can be seen by considering the symmetry of the characteristic polynomial of $A$.
 This is a different explanation for the symmetry of the Lyapunov spectrum of symplectic cocycles.

 \item There are some other symmetries the cocycle might have.
 If it preserves a symmetric bilinear form of signature $(p,q)$ with $p\geq q$, then the spectrum has the form
 \begin{align*}
  \lambda_1\geq \cdots \geq \lambda_q \geq 0 \cdots \geq 0 \geq -\lambda_q \geq \cdots \geq -\lambda_1
 \end{align*}
 In particular, there are at least $p-q$ zero exponents.
 
 \item The Oseledets theorem holds for both real and complex vector bundles (with hermitian metrics).
 The Lyapunov exponents are real numbers in both cases.
\end{enumerate}
\end{remark}

\subsection{Structure theory of Lie groups}

\subsubsection{Setup.}
Throughout this section $G$ will be a real semisimple Lie group with Lie algebra $\frakg$.
Fix a maximal compact subgroup $K\subset G$ with Lie algebra $\frakk$.
Then $G$ can be equipped with a \emph{Cartan involution} $\sigma:G\to G$ whose action on $\frakg$ gives the decomposition into eigenspaces
\begin{align*}
 \frakg = \frakk\oplus \frakp \textrm{ where }\sigma\vert_{\frakk}=+1 \textrm{ and }\sigma\vert_{\frakp}=-1
\end{align*}
Recall also the \emph{Killing form} on $\frakg$ given by $\ip{x,y}=\tr (\ad_x\circ \ad_y)$.
It is negative-definite on $\frakk$ and positive-definite on $\frakp$.

\subsubsection{Example of $\SL_n\bR$.}
Consider
\begin{align*}
 \frakg = \fraksl_n\bR =\{ a\in \Mat_{n\times n}\bR \vert \tr a =0 \}.
\end{align*}
A maximal compact subgroup is $K=\SO_n\bR$ with Lie algebra $\fraks\frako_n\bR$.
The involution $\sigma$ acts on $\SL_n\bR$ by $g\mapsto (g^{-1})^t$ and on $\frakg$ by $x\mapsto -x^t$.
The subspace $\frakp$ on which $\sigma$ acts by $-1$ is the space of symmetric matrices:
\begin{align*}
 \frakp = \{x\in \fraksl_n\bR \vert x=x^t\}
\end{align*}

\begin{remark}
 In fact, any real semisimple Lie algebra $\frakg$ admits an embedding $\frakg\into \fraksl_n\bR$ such that the Cartan involution on $\frakg$ is that of $\fraksl_n\bR$ restricted to $\frakg$.
 In this case $\frakk=\frakg \cap \fraks\frako_n\bR$ and similarly for $\frakp$.
\end{remark}

\subsubsection{Split maximal Cartan algebra and Polar Decomposition.}
Pick a maximal abelian subalgebra $\fraka\subset \frakg$.
By definition, any two elements of $\fraka$ commute, and $\fraka$ is maximal with this property.
This is called a \emph{split Cartan subalgebra}.
It carries an action of a reflection group whose description is omitted.
The reflection hyperplanes divide $\fraka$ into chambers; fix one such $\fraka^+\subset \fraka$, called a Weyl chamber.
Using the exponential map, $\fraka$ gives the Lie subgroup $A\subset G$; the chamber $\fraka^+$ determines a semigroup $A^+\subset A$.

The polar (or Iwasawa, or $KAK$) decomposition of an element $g\in G$ describes it as a product
\begin{align*}
 g=k_1 a k_2 \textrm{ where }k_i \in K \textrm{ and }a\in A^+.
\end{align*}
In other words $G=KA^+K$.
The decomposition of an element is typically, though not always, unique.

\subsubsection{The example of $\SL_n\bR$.}
The spectral theorem says that a symmetric $n\times n$ matrix $M$ has a basis of orthogonal eigenvectors.
Thus $M$ is conjugate to a diagonal matrix via an orthogonal transformation:
\begin{align*}
 M = k\cdot d\cdot  k^t \textrm{ where }k\in \SO_n\bR \textrm{ and }d \textrm{ is diagonal.}
\end{align*}
Because permutation matrices are orthogonal, one can assume the entries of $d$ are in increasing order.

Let now $g\in \SL_n\bR$ and consider the elements $gg^t$ and $g^tg$.
Both of them are symmetric, so the spectral theorem gives
\begin{align*}
 gg^t &= k_1 d_1 k_1^t\\
 g^t g & = k_2 d_2 k_2^t
\end{align*}
One can check that $d_1=d_2=d$ and then that $g=k_1 \sqrt{d} k_2$.
One can take the entries of $\sqrt{d}$ to be positive, at the expense of adjusting the signs of $k_1$ and $k_2$.

A maximal split Cartan subalgebra for $\fraksl_n\bR$ is given by diagonal matrices
\begin{gather}
\begin{align*}
 \fraka &= \{a = \diag(\lambda_1,\ldots,\lambda_n) \vert \lambda_1+\ldots+\lambda_n = 0\}
\end{align*}
\intertext{A Weyl chamber is given by}
\begin{align*}
 \fraka^+ &= \{a=\diag(\lambda_1,\ldots,\lambda_n)\in \fraka \vert \lambda_1\geq \cdots \geq \lambda_n\}
\end{align*}
\end{gather}

\subsubsection{Singular values}
\label{sssec:sing_values}
More intrinsically, suppose that $T:V\to W$ is a linear map, and each of $V$ and $W$ is equipped with a metric.
The superscript $(-)^\dag$ denotes the adjoint of an operator.

Consider the symmetric self-adjoint linear operator $T^\dag T:V\to V$.
By the spectral theorem, it can be diagonalized with positive eigenvalues
\begin{align*}
\sigma_1(T)^2\geq \cdots \geq \sigma_n (T)^2
\end{align*}
These are called the \emph{singular values} of $T$.
The top singular value also computes the operator norm of $T$:
\begin{align}
\label{eqn:sigma_1_op_norm}
\sigma_1(T)=\norm{T}_{op} := \sup_{0\neq v\in V} \frac{\norm{Tv} }{\norm{v} }
\end{align}

Moreover, one has the induced operators on exterior powers $\Lambda^k T : \Lambda^k V\to \Lambda^k W$.
Then the singular values of $\Lambda^k T$ can be expressed in terms of those of $T$, for example the largest one is given by
\begin{align*}
\sigma_1 (\Lambda^k T) = \sigma_1(T)\cdots \sigma_k(T)
\end{align*}
In particular, the operator norm of $\Lambda^k T$ gives the product of the first $k$ singular values.

\subsection{Symmetric spaces}

\subsubsection{Setup.}
Consider the quotient $X:=G/K$, equipped with a left $G$-action.
At the distinguished basepoint $e=\id K$, the tangent space is canonically identified with $\frakp$.
Here are some properties of the quotient:
\begin{enumerate}
 \item[(i)] It admits a canonical $G$-invariant metric, coming from the restriction of the Killing form to $\frakp$.
 
 \item[(ii)] The space $X$ is diffeomorphic to $\frakp$.
 The exponential map $\exp:\frakp \to G \to G/K$ exhibits the diffeomorphism. 
 
 \item[(iii)] The canonical metric has non-positive sectional curvature.
 At the basepoint, in the direction of normalized vectors $x,y\in \frakp$, it involves the commutator of $x$ and $y$ and is given by
 \begin{align*}
  K(x,y) = -\frac 12 \norm{[x,y]}^2
 \end{align*}
 In particular, if $x$ and $y$ commute, the curvature is $0$.
 The abelian subalgebra $\fraka$ determines, via the exponential map, a totally geodesic embedding of $\fraka$ (with Euclidean metric) into $X$.
\end{enumerate}

\subsubsection{The example of $\SL_2\bR$.}
Recall that the hyperbolic plane can be described as a quotient $\bH=\SL_2\bR/ \SO_2\bR$.
The canonical symmetric space metric will in this case be the hyperbolic metric.

\begin{remark}
 The action of $G$ on $X$ is transitive, i.e. for any $x,y\in X$ there is a $g\in G$ such that $gx=y$.
 This action also preserves distances.
 
 In the hyperbolic plane $\bH$, for any $x_1,x_2$ and $y_1,y_2$ such that $d(x_1,x_2)=d(y_1,y_2)$, there is a transformation $g\in G$ such that $g x_i = y_i$.
 
 However, for a general $X$ the action does not act transitively on pairs of points at the same distance.
 The number of parameters that needs to be ``matched'' is equal to the dimension of $\fraka$.
 In the case of $\SL_2\bR$, this dimension is $1$ and the parameter corresponds to distance.
\end{remark}

\subsubsection{Geodesics in $X$}
\label{sssec:geodesic_desc}
Given a point $x\in X = G/K$, using the decomposition $G=KAK$ it follows that $x=k_x a_x e$, where $k_x \in K$ is to be thought of as a ``direction'' and $a_x$ as a ``distance''.
All geodesic rays in $X$ starting at the basepoint $e$ can be parametrized by
\begin{align*}
\gamma(t) = k \exp(t\cdot \alpha)e \textrm{ where }k\in K, \alpha \in \fraka^+
\end{align*}
The geodesic ray is unit speed if $\norm{\alpha}^2=1$.
Note that if $\alpha$ is on a wall of the Weyl chamber (for $\fraksl_n\bR$, this means some eigenvalues coincide) then different $k\in K$ can give the same geodesic.
This is also the way in which the $KAK$ decomposition can fail to be unique.

\subsubsection{Cartan projection}
For $x\in X$ let $r(x)\in \fraka^+$ denote the unique element such that $k\exp(r(x))e=x$ for some $k\in K$.
Although $k$ might not be unique, the element $r(x)$ is.
It can be viewed as a ``generalized radius'' (in the literature, also called a Cartan projection).
Its norm in $\fraka$ is equal to the Riemannian distance from $e$ to $x$.

For example, start with a matrix $g\in G=\GL_n\bR$, with maximal compact $K=\operatorname{O}_n(\bR)$ and split Cartan the diagonal matrices $\fraka\subset \gl_n\bR$.
Consider the point $g\in G/K$ obtained by applying $g$ to the basepoint $e\in G/K$.
Then the Cartan projection $r(g)\in \fraka^+$ is the diagonal matrix with $i$\textsuperscript{th} entry $\log (\sigma_i(g)$) (the singular values $ \sigma_i $ are defined in \autoref{sssec:sing_values}).

\subsubsection{Regularity}
 \label{sssec:reg_sequence}
 A sequence of points $\{x_n\}$ in $X$ is \emph{regular} if there exists a geodesic ray $\gamma:[0,\infty)\to X$ and $\theta\geq 0$ such that
 \begin{align*}
  \dist(x_n,\gamma(\theta\cdot n)) = o(n)
 \end{align*}
 In other words the quantity $\frac 1n \dist(x_n,\gamma(\theta\cdot n))$ tends to zero.
 
\begin{remark}
 \leavevmode
 \begin{enumerate}
  \item[(i)] If the parameter $\theta$ is zero, the sequence satisfies $d(x_n,e)=o(n)$.
  Recall that $e\in X$ is the distinguished basepoint.
  \item[(ii)] A sequence $g_n \in G$ is \emph{regular} if the sequence of points $\{g_n e\}$ is regular.
 \end{enumerate}
\end{remark}

\subsubsection{Kaimanovich's characterization of regularity}
The regularity of a sequence in a symmetric space $X$ can be characterized by rather simple conditions.
The non-positive curvature assumption is crucial for this description to hold, and in other metric spaces such characterizations are not known.

\begin{theorem}[Kaimanovich \cite{Kaimanovich}]
\label{thm:reg_pts}
 A sequence of points $\{x_n\}$ in $X=G/K$ is regular if and only if the following two conditions are satisfied:
 \begin{custom_description}{Distances converge}
  \item[Small steps] $\dist(x_n,x_{n+1})=o(n)$
  \item[Distances converge] $\alpha:=\lim \frac{r(x_n)}{n}$ exists in $\fraka^+$.
 \end{custom_description}
\end{theorem}

A precursor to this result is due to Ruelle \cite[Prop. 1.3]{Ruelle_ergodic}, where it is phrased in terms of matrices.

\begin{remark}
 \leavevmode
 \begin{enumerate}
  \item [(i)] Recall that there exists a sequence of points $\{x_n\}\subset \bR^2$ with $\dist(x_n,x_{n+1})=O(n)$, $\lim \frac {|x_n|}n$ exists, but the angles of $x_n$ to the origin don't converge.
  One can take the points to be on a logarithmic spiral.
  This example is in the ``zero curvature'' case.
  
  Totally geodesic embeddings of $\bR^d$ into $G/K$ arise using the exponential map in $\fraka$, and for this reason the convergence of the vector-valued distance in $\fraka^+$ is necessary in \autoref{thm:reg_pts}.
  
  \item [(ii)] It is instructive to verify \autoref{thm:reg_pts} in the case when $X$ is a tree, and $r(x)$ is simply the distance to a fixed basepoint.
 \end{enumerate}
\end{remark}

\noindent \textit{Proof of \autoref{thm:reg_pts}.}
 For simplicity, consider the case of the hyperbolic plane $\bH=\SL_2\bR/\SO_2\bR$, equipped with the metric of constant negative curvature.
 This case contains the main idea and is easier notationally.

 First, recall the Law of Sines in constant negative curvature.
 Consider a triangle with angles of sizes $\alpha,\beta,\gamma$ and opposite sides of lengths $a,b,c$.
 These quantities are related by
 \begin{equation}
  \label{eqn:sine_law}
  \frac{\sin \alpha}{\sinh a} = \frac{\sin \beta}{\sinh b} = \frac{\sin \gamma}{\sinh c}
 \end{equation}
 Therefore $\sin \alpha = \sin \beta \cdot \frac{\sinh a}{\sinh b}$.
 Recall that we have the estimates
 \begin{align*}
 \begin{split}
  \dist(x_n,x_{n+1}) &= o(n)\\
  \dist(e,x_n) &= \theta n + o(n)
 \end{split}
 \end{align*}
 Assume that $\theta>0$, otherwise the claim follows directly.


 
 Define the angle, when viewed from the origin, between successive points: $\phi_n:=\measuredangle(x_n e x_{n+1})$.
 Define also the angle $\beta_{n+1}:= \measuredangle(x_n x_{n+1} e)$.
 
 The law of sines then gives
 \begin{align*}
  \sin \phi_n = \sin \beta_{n+1} \cdot \frac{\sinh (\dist (x_n,x_{n+1}))}{ \sinh (\dist(e,x_{n+1}) ) }
 \end{align*}
 \begin{wrapfigure}[10]{r}{0.4\textwidth}
 	\includegraphics[width=0.38\textwidth]{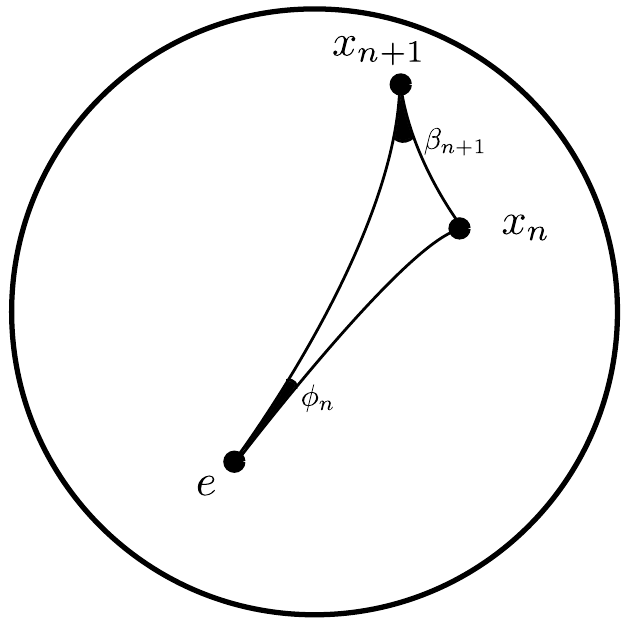}
 	\label{fig:reg_seqn}
 \end{wrapfigure}
 \leavevmode

 Applying the bounds $\frac 1 {10}|x|\leq |\sin x|$ (for $|x|\leq \pi/10$) and $|\sin x|\leq 1$ to $\phi_n$ and $\beta_{n+1}$ respectively, combined with the assumed bounds on distance, yields
 \begin{wrapped}{align*}
   |\phi_n| &\leq 10 \cdot \frac{\sinh ( \dist(x_n,x_{n+1}) ) }{\sinh ( \dist (e,x_{n+1}) ) } \leq \frac{e^{o(n)}}{e^{\theta n + o(n)} } \\ \nonumber
  &\leq e^{-\theta n +o(n)}
 \end{wrapped}
 Therefore, the total angle $\phi:=\sum_{n\geq 0} \phi_n$ converges uniformly.
 The geodesic launched at angle $\phi$ and speed $\theta$ will be $o(n)$-close to $x_n$ (again, by an application of the law of sines).
\hfill \qed

\begin{remark}
 \leavevmode
 \begin{enumerate}
 \item[(i)]
 To generalize the above proof to all symmetric spaces $X=G/K$, one needs two further ingredients.
 The first is a comparison theorem for triangles in manifolds with all sectional curvatures bounded by $ \kappa\leq 0$.
 The second is a more detailed use of the structure theory of Lie groups.
 \item[(ii)]
 A general symmetric space $X=G/K$ will contain geodesically embedded copies of Euclidean spaces $\bR^{\dim \fraka}$, called flats.
 These arise from taking the exponential map of maximal abelian subalgebras inside $\frakp$.
 If the sequence of points $\{x_n\}$ is contained in such a flat, the condition on convergence of vector-valued distance from \autoref{thm:reg_pts} is both necessary and sufficient for the existence of the geodesic.
 \end{enumerate}
\end{remark}

\subsection{Symmetric spaces and the Oseledets theorem}

An introduction to the formalism used below is available in the book of Zimmer \cite{Zimmer}.
For this section we consider more general real reductive Lie groups $G$.
This class includes $\GL_n\bR$, not just $\SL_n\bR$, so it allows factors such as $\bR^\times$.

\noindent For most considerations, including symmetric spaces and the Oseledets theorem, everything can be reduced to semisimple groups.
Indeed, for a cocycle in $\GL_n\bR$, the maps can be rescaled to assume the cocycle acts by matrices with determinant $ \pm 1 $.

\subsubsection{Setup.}
Recall that $T:(\Omega,\mu)\to (\Omega,\mu)$ is an ergodic probability measure preserving transformation and $E\to \Omega$ is a vector bundle equipped with a metric $\norm{-}$.
The extension of $T$ to a cocycle on $E$ is a collection of linear maps between the fibers of $E$:
 \begin{align}
  T_\omega : E_\omega \to E_{T\omega}
 \end{align}
If $G$ is a reductive Lie group, what does it mean to say that the maps $T_\omega$ belong to this group?
After all, the maps are between distinct vector spaces.

\subsubsection{Principal bundles.}
Suppose that $E\to \Omega$ is a rank $n$ vector bundle.
For each fiber $E_\omega$ consider the set of isomorphisms to a fixed vector space
 \begin{align}
  P_\omega := \Isom (\bR^n,E_\omega)
 \end{align}
Note that $P_\omega$ carries a right action of $\GL_n\bR$ by precomposing the isomorphism.
This action makes it isomorphic to $\GL_n\bR$, except that it does not have a distinguished basepoint.
The spaces $P_\omega$ glue to give a fiber bundle $P\to \Omega$.

More generally, if $E$ has some extra structure, e.g. a symplectic form, then $P_x$ can be the set of isomorphisms respecting the extra structure on the source and target.
It will be isomorphic to a subgroup of $\GL_n\bR$, e.g. the symplectic group if $E$ carries a symplectic form.

By definition, a principal $G$-bundle over $\Omega$ is a space $P$ with a map $P\xrightarrow{\pi} \Omega$ and a right action of $G$ on $P$ such that $\pi(pg)=\pi(p)$.
Moreover, each fiber $P_\omega$ must be isomorphic to $G$ with the right $G$-action.

\subsubsection{Induced bundles.}
Let $P\to \Omega$ be a principal $G$-bundle, and suppose that $G$ acts on another space $F$.
The associated bundle over $\Omega$ with fiber $F$ is the quotient
 \begin{align*}
  P \times_G F :=\raisebox{.2em}{$\{(p,f)\in P\times F \}$ } / \raisebox{-.2em}{$(p,f) \sim (pg,g^{-1}f)$}
 \end{align*}

\subsubsection{Cocycles on principal bundles.}
Suppose that $\Omega$ carries a $T$-action and a principal $G$-bundle $P\xrightarrow{\pi} \Omega$.
Then a cocycle $T:P\to P$ is a lift of the $T$-action from $\Omega$ to $P$ which commutes with the $G$-action of $P$ on the right.

If $G$ acts on a space $F$, then the cocycle $T$ on $P$ extends to a natural action of $T$ on $P\times_G F$.
For instance, if $G$ has a representation on $F=\bR^n$, then $P\times_G \bR^n$ is a vector bundle with a linear cocycle.

\begin{example}
 Suppose that $G=\GL_n(\bR)$ and $K=\operatorname{O}_n(\bR)$.
 Then $X=G/K$ is the space of metrics (coming from positive-definite inner products) on $\bR^n$, with distinguished basepoint the euclidean metric (corresponding to the coset $e K$).
 Call this the standard metric $\norm{-}_{std}$.
 Then for $x=gK\in G/K$ the metric is given by $\norm{v}_x = \norm{g^{-1}v}_{std}$.
 
 Now, a vector bundle $E\to \Omega$ gives rise to a principal $G=\GL_n\bR$-bundle $P\to \Omega$.
 For $X:=G/K$ consider the induced bundle $\crX:=P\times_G X \to \Omega$.
 A fiber $\crX_\omega$ is the space of metrics on $E_\omega$.
 Thus, a metric on $E\to \Omega$ is the same as a choice of point in each fiber of $\crX\xrightarrow{\pi} \Omega$, i.e. a map $\sigma:\Omega \to \crX$ such that $\pi\circ \sigma(\omega) = \omega$.
 
 If $E$ is a cocycle over $T:\Omega\to \Omega$, the action of $T$ extends to $P$ and $\crX$.
 Typically it will \emph{not} preserve the metric on $E$ but rather give an action $\sigma \mapsto T\sigma$ on the space of metrics; for $v\in E_\omega$ the new metric is defined by the action of $T$:
 \begin{align*}
  \norm{v}_{T\sigma(\omega)} := \norm{Tv}_{\sigma(T\omega)}
 \end{align*}
\end{example}

\begin{proposition}
\label{prop:Oseledets_equiv_regularity}
 Consider the sequence of vector spaces and linear maps
 \begin{align*}
  V_\omega \xrightarrow{T} V_{T\omega} \xrightarrow{T} V_{T^2\omega}\to \cdots
 \end{align*}
 Recall that each space $V_{T^\bullet\omega}$ carries a metric $\norm{-}$.
 Then the following statements are equivalent:
 \begin{enumerate}
  \item [(i)] The Oseledets theorem holds for $E_\omega$, i.e. there exist numbers $\lambda_1>\cdots > \lambda_k \neq -\infty$ and a filtration $V^{\leq \lambda_i}$ such that each $v\in V^{\leq \lambda_i}\setminus V^{\leq \lambda_{i+1}}$ has the asymptotic growth
  \begin{align*}
   \lim_{N\to \infty} \frac{1}{N}  \log \norm{T^N v} = \lambda_i
  \end{align*}
  \item [(ii)] The sequence of metrics $\norm{-}_N$ on $V_\omega$ defined by $\norm{v}_N:=\norm{T^N v}$ is a regular sequence (see \autoref{sssec:reg_sequence}) in the symmetric space $\GL(V_\omega)/\operatorname{O}(V_\omega)$.
  Here $\operatorname{O}(V_\omega)$ is the group of orthogonal transformations preserving the initial metric on $V_\omega$.
  \item [(iii)] There exists an invertible liner map $\Lambda:V_\omega \to V_\omega$ which is symmetric and self-adjoint (for the fixed metric on $V_\omega$) and such that $\forall v\in V_\omega$ satisfies
  \begin{align*}
   |\ip{ \Lambda^{-2n} ({T^n})^\dag T^n v,v }| = o(n)
  \end{align*}
  Above $({T^n})^{\dag}$ denotes the adjoint for the initial metrics on $V_\omega$ and $V_{T^n\omega}$.
 \end{enumerate}
\end{proposition}
\begin{proof}
 For the equivalence of (i) and (iii), assume first (i).
 Given the Oseledets filtration $V^{\leq \lambda_j}_\omega$, define $V^{\lambda_j}_\omega$ as the orthogonal complement of $V^{\lambda_{j+1}}_\omega$ inside $V^{\leq \lambda_j}_\omega$.
 These spaces give a direct sum decomposition of $V_\omega$.
 Declare $\Lambda$ to act as the scalar $e^{\lambda_j}$ on $V^{\lambda_j}_\omega$.
 Then $\Lambda$ satifies (iii).
 
 Conversely, given $\Lambda$, define $V^{\lambda_j}_\omega$ using the eigenspaces of $\Lambda$ and $V^{\leq \lambda_j}_\omega$ using the partial sums of $V^{\lambda_j}_\omega$.
 The asymptotic behavior guaranteed by (iii) shows that $V^{\leq \lambda_j}_\omega$ satisfies the properties of the Oseledets filtration.
 
 Finally, the equivalence of (ii) and (iii) follows from the definition of a regular sequence in \autoref{sssec:reg_sequence} and the description of a geodesic in \autoref{sssec:geodesic_desc}.
\end{proof}

\begin{remark}
 \autoref{prop:Oseledets_equiv_regularity} shows the Oseledets theorem for the cocycle $V\to \Omega$ is equivalent to a statement on the associated bundle of symmetric spaces $\crX:=P\times_G X$ where $X=G/K$ and $G=\GL_n\bR, K=\SO_n\bR$.
 
 Namely, we have the linear maps $T^N:V_\omega \to V_{T^N\omega}$ and the initial norms $\norm{-}$ on the corresponding spaces.
 The dynamics defines new norms $\norm{v}_N:=\norm{T^N v}$ on $V_\omega$ and it suffices to show that for a.e. $\omega\in \Omega$, this sequence of norms is regular in the symmetric space $\crX_\omega$.
 
 However, Kaimanovich's \autoref{thm:reg_pts} gives a simple criterion to check regularity of a sequence (see \autoref{thm:Oseledets_geometric} below).
\end{remark}

\subsection{Geometric form of the Oseledets theorem}

\begin{theorem}
\label{thm:Oseledets_geometric}
 Let $E\to \Omega$ be a cocycle over an ergodic measure-preserving transformation $T:(\Omega,\mu)\to (\Omega,\mu)$.
 Suppose that $E$ carries a metric $\norm{-}$ such that (see \eqref{eqn:cond_L1_bdd_cocycle})
 \begin{align*}
  \int_\Omega \log^+ \norm{T_\omega}_{op} d\mu(\omega) <\infty 
  \textrm{ and }
  \int_\Omega \log^+ \norm{T_\omega^{-1}}_{op} d\mu(\omega) <\infty
 \end{align*}
 Consider the associated symmetric space bundle $\crX$ whose fiber over $\omega\in\Omega$ is the space of metrics on $E_\omega$.
 Then the sequence of metrics defined by $\norm{v}_N:=\norm{T^N v}$ is a regular sequence in $\crX_\omega$, for $\mu$-a.e. $\omega \in \Omega$.
\end{theorem}
From \autoref{prop:Oseledets_equiv_regularity}, the above theorem is equivalent to the usual form of the \hyperref[thm:Oseledets]{Oseledets Theorem \ref*{thm:Oseledets}}.
The proof of \autoref{thm:Oseledets_geometric} will make use of the Subadditive Ergodic Theorem, recalled below.
A proof is available, for example, in any of \cite{Ledrappier_lectures, Mane, Viana}.

\begin{theorem}[Kingman Subadditive Ergodic Theorem]
\label{thm:Kingman}
 Let $T\curvearrowright(\Omega,\mu)$ be an ergodic probability measure-preserving transformation.
 Suppose that $\{f_i\}$ is a sequence of functions on $\Omega$ with $f_1\in L^1(\Omega,\mu)$ and satisfying the subadditivity condition
 \begin{align}
 \label{eqn:subadd}
  f_i(\omega) + f_j(T^i\omega) \geq f_{i+j}(\omega) \hskip 1em \forall \omega\in \Omega, \hskip 1em \forall i,j\geq 1
 \end{align}
 Then for $\mu$-a.e. $\omega$ the limit $\frac 1N f_N(\omega)$ exists (perhaps $ -\infty $) and can be computed as
 \begin{align*}
  \lim_{N\to \infty} \frac 1N f_N(\omega) = \inf_N \frac 1N \int_{\Omega} f_N(\omega) d\mu(\omega)
 \end{align*}
\end{theorem}
Note that integrating the subadditivity condition \eqref{eqn:subadd} it follows that in the statement of the Subadditive Ergodic Theorem, each $f_i$ is bounded above by a function in $L^1(\Omega,\mu)$.
Therefore each $f_i$ is the sum of a function in $L^1$ and an everywhere negative function.

\subsubsection{Fekete's lemma}
\label{sssec:linear_drift}
Let $\{a_n\}$ be a subadditive sequence, i.e. assume that
\begin{align*}
a_{n+m} \leq a_n + a_m \hskip 1em \forall n,m\geq 1
\end{align*}
It is a classical exercise that the sequence $\frac 1N a_N$ has a limit given by the infimum:
\begin{align*}
\lim_N \frac 1N a_N = \inf_N \frac 1N a_N
\end{align*}
which could potentially be $ -\infty $.

\begin{proof}[Proof of \autoref{thm:Oseledets_geometric}]
 Define the sequence of functions 
 \begin{align*}
  f_N(\omega):=\log \norm{T^N_\omega}_{op}
 \end{align*}
 The operator norms are computed for the initial norm $\norm{-}$ on $E_\omega$ and $E_{T^N\omega}$.
 
 Since for two linear maps $\norm{A\circ B}_{op} \leq \norm{A}_{op}\cdot \norm{B}_{op}$, the sequence $\{f_i\}$ satisfies the subadditivity condition \eqref{eqn:subadd} in the Kingman Theorem.
 Since $f_1\in L^1(\Omega,\mu)$, \autoref{thm:Kingman} applies and so there exists $\lambda_1$ such that
 \begin{align}
  \lim \frac 1N \log \norm{T^N_\omega}_{op} = \lambda_1
 \end{align}
 Finally, recall that the operator norm is the same as the first singular value of $T_\omega$ (see \eqref{eqn:sigma_1_op_norm}).
 Thus the quantity $\frac 1N \log (\sigma_1(T^N_\omega))$ converges $\mu$-a.e.
 
 Apply now the same construction to the exterior power bundles $\Lambda^k V$ to find that in fact all (normalized) singular values converge.
 Finally, recall that the Cartan projection $r(x_N)$ was given by considering the singular values of the corresponding operator, and so $\frac 1N r(x_N)$ converge (where $x_N$ denote the pull-back metrics after $N$ steps).
 
 To check the small steps condition in \autoref{thm:reg_pts}, recall that if $f\in L^1(\Omega,\mu)$ then by the Birkhoff \autoref{thm:Birkhoff} (see \eqref{eqn:L1_basic_bound}) it follows that
 \begin{align}
  \textrm{for }\mu\textrm{-a.e. }\omega\in \Omega \hskip 1em \frac 1 N f(T^N \omega) \to 0
 \end{align}
 Setting $f(\omega)=\log^+ \norm{T_\omega}_{op}$ and $ f(\omega)=\log^+ \norm{T_\omega^{-1}}_{op} $ gives the desired bound on the increments of the highest and lowest singular values, and hence on the intermediate ones.
\end{proof}

\begin{remark}
	The integrability assumption on $ T^{-1}_\omega $ in \autoref{thm:Oseledets_geometric} is not strictly necessary.
	In order to have a formulation without it, one needs to extend the notion of regular sequence to allow for super-linear divergence in certain ``flat'' directions in the symmetric space.
	Note that in the case of a cocycle valued in $ \SL_n(\bR) $, the integrability condition for $ T_\omega $ implies the one for $ T^{-1}_{\omega} $.
\end{remark}

\section{The general Noncommutative Ergodic Theorem}
\label{sec:general_ncet}

This lecture follows closely the notes of Karlsson \cite{Karlsson_notes}.
The main result in the case of isometries was proved by Karlsson--Ledrappier \cite{Karlsson_Ledrappier} and recently extended by Gou\"{e}zel--Karlsson to semi-contractions \cite{Gouezel_Karlsson}.

Horofunctions are introduced in \autoref{ssec:horofunctions}; the case of divergence to infinity along geodesics, i.e. the case of Busemann functions, is illustrated in \autoref{ssec:buseman_fctn}.
A very general form of a Noncommutative Ergodic Theorem is stated in \autoref{ssec:ncet} and proved in \autoref{ssec:proof_ncet}.
It is illustrated with a non-trivial example in \autoref{ssec:example}.

\subsection{Horofunctions}
\label{ssec:horofunctions}

\subsubsection{Setup.}
Let $(X,d)$ be a metric space which is proper, i.e. balls of bounded radius are compact.
Let $C^0(X)$ be the space of continuous functions on $X$, with the $\sup$-norm.
Fix a basepoint $x_0\in X$.
This defines an embedding
\begin{align*}
 \Phi :  X &\to C^0(X)\\
 x &\mapsto \Phi(x)(y)=d(x,y)-d(x,x_0)
\end{align*}
To simplify notation, $\Phi(x)$ will also be denoted by $h_x$.
\begin{remark}
 \leavevmode
 \begin{enumerate}
  \item[(i)] The functions $h_x$ are $1$-Lipschitz, since they are given by distance to $x$ with a constant subtracted.
  \item[(ii)] The map $\Phi$ itself is $1$-Lipschitz, since
  \begin{align*}
   |\Phi(x)(y) - \Phi(x)(z)| = |d(x,y)-d(x,z)| \leq d(y,z)
  \end{align*}
 \item[(iii)] The map is normalized to have $h_x(x_0)=0$ and it is moreover injective.
 Indeed, if $d(x,x_0)\geq d(y,x_0)$ then $\Phi(y)(x)-\Phi(x)(x)\geq d(x,y) >0$.
 \end{enumerate}
\end{remark}
\begin{definition}
\label{def:metric_bord}
The \emph{metric bordification} of $X$ is its closure inside $C^0(X)$:
\begin{align*}
 \conj{X} := \conj{\Phi(X)} = X \cap \partial X.
\end{align*}
If $X$ is proper, then $\conj{X}$ is (sequentially) compact by Arzela--Ascoli and because the functions $h_x$ are $1$-Lipschitz.
Functions in $\partial X$ are called \emph{horofunctions} on $X$.
\end{definition}

\begin{example}
\leavevmode
 \begin{enumerate}
  \item If $X=\bR^2$ with Euclidean distance, then $\partial \bR^2$ equals the linear functions of norm $1$.
  \item If $\bD^2\cong \bH^2$ is the hyperbolic plane, then the boundary is isomorphic to $\bR\bP^1$.
  In the upper halfplane model, corresponding to $\infty\in \partial \bH$ is the function $-\log y$ (with $x_0 = \sqrt{-1}\in \bH$).
 \end{enumerate}
\end{example}

\begin{proposition}
 The action of the isometry group of $X$ extends continuously to an action on $\conj{X}$.
 Given an isometry $g$ and a function $h\in \conj{X}$, the action is by the formula
 \begin{align}
 \label{eqn:isom_act_horo}
  g\cdot h(z) := h(g^{-1}z)-h(g^{-1}x_0)
 \end{align}
\end{proposition}
\begin{proof}
 The action described in \eqref{eqn:isom_act_horo} is continuous on the space of functions, so it suffices to check its compatibility with the action on $X$.
 Suppose that $x\in X$ has corresponding function $h_x$.
 We have the chain of equalities
 \begin{align}
  \nonumber h_{gx}(z) &= d(gx,z) - d(gx,x_0)\\
  \nonumber &= d(x,g^{-1}z) - d(x,x_0) - (d(x,g^{-1}x_0) - d(x,x_0))\\
  \nonumber &= h_x(g^{-1}z) - h_x(g^{-1}x_0)
 \end{align}
 This shows the compatibility of actions.
\end{proof}

\begin{remark}
 Up to homeomorphism, the space $\conj{X}$ is independent of the choice of basepoint $x_0\in X$.
 Consider the quotient $C^0(X) \to C^0(X)/\{const.\}$ of the space of continuous functions by the constant ones.
 Then the image of $\Phi(X)$ under the quotient is independent of the basepoint $x_0$, and is still an embedding.
\end{remark}

\subsubsection{Semi-contractions}
\label{sssec:semi-contractions}
Let $f$ be a semi-contraction of $X$, i.e. $d(fx,fy)\leq d(x,y)$ for all $x,y\in X$.
Consider the quantity $d(x_0,f^n x_0)$ and apply the triangle inequality with the semi-contraction property to find
\begin{align}
 \nonumber d(x_0,f^{n+m}x_0) &\leq d(x_0,f^n x_0) + d(f^n x_0, f^{n+m}x_0)\\
 \nonumber & \leq d(x_0,f^n x_0) + d(x_0, f^m x_0)
\end{align}
Thus the limit of $\frac 1N d(x_0,f^N x_0)$ exists, and is called the \emph{linear drift} of the semi-contraction $f$.

In fact, the drift can be detected by a single horofunction.

\begin{proposition}[Karlsson]
\label{thm:Karlsson_drift}
 Let $(X,d)$ be a proper metric space and $f$ a semicontraction with linear drift $l$.
 Then there exists $h\in \conj{\Phi(X)}$ such that
 \begin{align*}
 \begin{split}
  &\forall k\geq 0 \hskip 1em h(f^{k}x_0) \leq -l\cdot k \\
  &\forall x\in X \hskip 1em \lim_N \frac {-1}N h(f^N x) = l \label{eqn:lim_all_x}
  \end{split}
 \end{align*}
\end{proposition}
\begin{exercise}
\label{ex:subseq}
 Suppose $a_{n+m}\leq a_n + a_m$ and $\lim \frac 1N a_N = l$.
 Then there exists a subsequence $\{m_j\}$ of $1,2,3,\ldots$ such that for any further subsequence $\{n_i\}$ of $\{m_j\}$ we have for all $k\geq 0$:
 \begin{align*}
  \liminf_{n_i} \left(a_{n_i - k} - a_{n_i}\right) \leq -l\cdot k
 \end{align*}
\end{exercise}
\begin{proof}[Proof of \autoref{thm:Karlsson_drift}]
 Set $a_n:=d(x_0,f^n x_0)$ and pick the subsequence $\{m_j\}$ provided by \autoref{ex:subseq}.
 Since $\conj{X}$ is compact, pick a further subsequence $\{n_i\}$ such that $f^{n_i}x_0$ converge to an element $h\in \conj{X}$.
 
 Next, observe that $-a_k \leq h(f^k x_0 )$ since from the definition of $h$ we have
 \begin{align*}
  h(f^k(x_0)) & = \lim_{n_i} d(f^{n_i}x_0,f^k x_0) - d(f^{n_i}x_0,x_0) \geq -d(x_0,f^k x_0)
 \end{align*}
 Using the definition of $h$ again, we have
 \begin{align*}
  \nonumber h(f^k(x_0)) &= \lim_{n_i} d(f^{n_i}x_0,f^k x_0) - d(f^{n_i}x_0,x_0) \\
  \nonumber & \leq \liminf_{n_i} \left(a_{n_i-k} - a_{n_i} \right)
 \end{align*}
 It follows that
 \begin{align*}
  -a_k \leq h(f^k(x_0)) \leq - l\cdot k
 \end{align*}
 and this gives the first part.
 Moreover, since $\frac 1k h(f^k x_0) \to l$ and $d(f^k(x_0),f^k(x))$ stays bounded, it follows the limit \eqref{eqn:lim_all_x} holds for all $x\in X$.
\end{proof}

\subsection{Busemann functions}
\label{ssec:buseman_fctn}

For a general metric space, horofunctions can give too large of a compactification.
For geodesic metric spaces, a subclass of horofunctions is distinguished as coming from geodesic rays.
These notions are introduced and illustrated below, and taken up again in \autoref{ssec:cat0_spaces}.

\subsubsection{Geodesics}
An interval $[0,r]\subset \bR$ of unspecified length will be denoted $I$ and a \emph{geodesic segment} inside $X$ will mean an isometric embedding $\gamma:I\to X$.
If the endpoints of the geodesic segment are $x=\gamma(0)$ and $y=\gamma(r)$, its image is denoted $[x,y]$.
The \emph{metric space} is \emph{geodesic} if there is at least one geodesic between any two points.
A \emph{geodesic ray} is an isometric embedding of $[0,+\infty)$ into $X$.

\begin{definition}
\label{def:buseman_fctn}
Let $\gamma:[0,\infty) \to X$ be a geodesic ray starting at $x_0$.
Define
\begin{align*}
 h_\gamma(x) := \lim_{t\to \infty} d(\gamma(t),x)-t
\end{align*}
Equivalently, one can take the limit of $\Phi(\gamma(t))$ in the compactification $\conj{X}$ in \autoref{def:metric_bord}.
The function $h_\gamma$ obtained this way is called a \emph{Busemann function}.
\end{definition}

\subsubsection{$\CAT(0)$ spaces}
\label{sssec:cat0}
Let $(X,d)$ be a geodesic metric space.
It is called a $\CAT(0)$ space if the following holds for all triples of points $x,y,z\in X$.

Connect the points $x,y,z$ by geodesics, and pick two points $p,q$ on distinct geodesics.
In $\bR^2$, there is a unique up to isometry triangle with vertices $x',y', z'$ which has the same side lengths as the triangle formed by $x,y,z$.
There is a unique map preserving lengths between the sides of $x,y,z$ and $x',y',z'$ and let $p',q'$ be the images of $p,q$.
Then we must have
\begin{align*}
 d_X(p,q) \leq d_{\bR^2}(p',q')
\end{align*}
A large class of examples of $ \CAT(0) $-spaces comes from complete, simply-connected manifolds with non-positive sectional curvature.
For instance $ \SL_n(\bR)/\SO_n(\bR) $ satisfies the requirements.

Busemann functions give natural examples of horofunctions, and for $\CAT(0)$ spaces, any boundary point of the metric bordification $\partial X \subset \conj{X}$ comes from a Busemann function $h_\gamma$ for some geodesic ray $\gamma$.
See also \autoref{sssec:boundaries} for more on the equivalence of boundaries for $ \CAT(0) $ spaces.

\subsubsection{Gromov-hyperbolic spaces}
A geodesic metric space $X$ is $\delta$-hyperbolic if its triangles are $\delta$-thin, i.e. for any $x,y,z\in X$ we have 
\begin{align*}
 [x,y]\subset N_\delta \left([x,z]\cup [z,y]\right)
\end{align*}
where $N_\delta(-)$ is the $\delta$-neighborhood of a set.

Unlike the $\CAT(0)$-property which holds at all scales, and in particular locally, $\delta$-hyperbolicity is coarse and does not restrict the local geometry of a space.
The metric (or horocycle) bordification $\partial_{hor} X$ of a $\delta$-hyperbolic space is in general too large.
A more natural boundary is formed by quasi-isometry classes of rays $\partial_{ray}X$ and there is a natural map $\partial_{hor} X \to \partial_{ray} X$.
Moreover, if two different $h,h'\in \partial_{hor}X$ have the same image in the ray compactification, then $h-h'$ is bounded by a constant uniformly on $X$.
Nevertheless, ergodic theorems such as \autoref{thm:NCET} can be developed in this setting.

\subsubsection{Busemann functions for $\SL_n\bR/\SO_n\bR$}

Consider the space $\frakp := \{M\in \Mat_{n\times n}(\bR) \vert M = M^t, \tr M=0\}$.
Recall we also have the symmetric space
\begin{multline*}
 X = \SL_n\bR / \SO_n \bR = \\
 =P_n(\bR):= \{M\in \Mat_{n\times n}(\bR)\vert M = M^t, \det M =1\}
\end{multline*}
The exponential map takes $\frakp$ diffeomorphically to $P_n(\bR)$.
Note that $\SL_n\bR$ acts on $P_n(\bR)$ by $g\cdot M = g M g^t$.

Pick now $\alpha\in \frakp$; after a conjugation by an orthogonal matrix, we can assume that $\alpha$ is diagonal, with eigenvalues occurring with multiplicity $(n_1,\cdots,n_k)$.

Using the partition $(n_1,\cdots,n_k)$ of $n$ define
\begin{align*}
 F(\alpha) & = 
 \begin{bmatrix}
  P_{n_1}(\bR) & 0 &\cdots & 0\\
  0 & P_{n_2}(\bR) & \cdots & 0\\
  \vdots & \vdots & \ddots & \vdots\\
  0 & \cdots & 0 & P_{n_k}(\bR)
 \end{bmatrix}
 && \parbox{1.2in}{ a product of symmetric spaces} \\
 N_\alpha & = 
 \begin{bmatrix}
  \id_{n_1}(\bR) & * &* & *\\
  0 & \id_{n_2}(\bR) & * & *\\
  \vdots & \vdots & \ddots & *\\
  0 & \cdots & 0 & \id_{n_k}(\bR)
 \end{bmatrix} && \parbox{1.2in}{ a unipotent subgroup}
\end{align*}
This gives a set of coordinates via the isomorphism
\begin{align}
 P_n(\bR) \cong N_\alpha \cdot F(\alpha)
\end{align}
The Busemann function corresponding to $\alpha$ is then
\begin{align}
 h_{\exp(\alpha\cdot t)}(n\cdot f) := -\tr (\alpha \cdot \log f)
\end{align}
Note that this function is constant on $N_\alpha$-orbits.

\subsection{Noncommutative Ergodic Theorem}
\label{ssec:ncet}

The next theorem was proved by Karlsson--Ledrappier \cite{Karlsson_Ledrappier}.
\begin{theorem}
\label{thm:NCET}
 Assume that $T\curvearrowright(\Omega,\mu)$ is a probability measure-preserving ergodic system.
 Let $(X,d)$ be a metric space with chosen basepoint $x_0\in X$ and isometry group $\Isom(X,d)$.
 
 Let $\Omega\to \Isom(X,d)$ be a measurable map, and denote by $g_\omega$ the isometry corresponding to $\omega\in \Omega$.
 Assume that we have the $L^1$-bound
 \begin{align*}
  \int_\Omega d(g_\omega x_0,x_0)d\mu(\omega) < \infty
 \end{align*}
 Then we have a linear drift defined by
 \begin{align}
 \label{eqn:noncomm_thm_drift}
  l:=\inf_N \frac 1N \int_{\Omega} d(g_\omega \cdots g_{T^{N-1}\omega} x_0, x_0) d\mu(\omega)
 \end{align}
 In addition, for $\mu$-a.e. $\omega$ we have
 \begin{align}
 \label{eqn:noncomm_thm_local_drift}
  \lim_{N} \frac 1 N d(g_\omega \cdots g_{T^{N-1}\omega} x_0,x_0) = l
 \end{align}
 Moreover, there exists a measurable map
 \begin{align*}
  \Omega & \to \conj{X}\\
  \nonumber \omega &\mapsto h_\omega
 \end{align*}
 such that for $\mu$-a.e. $\omega$ we have
 \begin{align*}
  \lim_{N\to \infty}\frac {-1}N h_\omega(g_\omega \cdots g_{T^{N-1}\omega} x_0) = l
 \end{align*}
 and if $ l>0 $ then $ h_\omega \in \partial X $ for a.e. $ \omega $.
\end{theorem}
\begin{remark}
 In general the map $\omega \mapsto h_\omega$ need not be equivariant, i.e. $h_{T\omega}$ need not equal $g_\omega \cdot h_\omega$.
 However, in many circumstances, e.g. when $X$ is $\CAT(0)$ or $\delta$-hyperbolic, this equivariance property can be arranged in the appropriate boundary.
\end{remark}

\subsection{Proof of the Noncommutative Ergodic Theorem}
\label{ssec:proof_ncet}

The following lemma is quite general and explains the mechanism in the proof of \autoref{thm:NCET}.

\begin{lemma}
\label{lemma:maximizing_measure}
 \leavevmode
 \begin{enumerate}
  \item Suppose $X$ is a compact metric space equipped with a continuous map $T:X\to X$.
  Let $f:X\to \bR$ be a continuous function and define
  \begin{align*}
   f_n(x) & := f(x) + \cdots + f(T^{n-1}x)
   \intertext{ and the quantity }
   a_n & :=\sup_{x\in X} f_n(x)
  \end{align*}
  Then $a_{n+m}\leq a_n + a_m$ and so $\lim_N \frac 1N a_N=l$ exists.
  Moreover, there exists a $T$-invariant measure $\mu$ on $X$ such that $\int_X f d\mu \geq l$.
  
  \item Suppose $X$ is a compact metric space and $T:(\Omega,\mu)\to (\Omega,\mu)$ is an ergodic probability measure preserving system.
  Let $S:\Omega\times X \to \Omega \times X$ be of the form $S(\omega,x):=(T\omega,g_\omega x)$ where $g_\omega$ is a homeomorphism of $X$ for $\mu$-a.e. $\omega$.
  Let also $f:\Omega\times X \to \bR$ be a function such that $f(\omega,-):X\to \bR$ is continuous for $\mu$-a.e. $\omega$.
  
  Define the Birkhoff averages
  \begin{align*}
   f_n(\omega,x) := f(\omega,x) + f(S(\omega,x)) + \cdots + f(S^{n-1}(\omega,x))
  \end{align*}
  and their fiberwise supremum $F_n(\omega):=\sup_{x\in X} f_n(\omega,x)$.
  
  Then $F_{n+m}(\omega)\leq F_n(\omega) + F_m(T^n\omega)$, so the sequence satisfies the assumptions of the Kingman \autoref{thm:Kingman}.
  Therefore $\frac 1N F_N(\omega)$ tends to a limit $l$ for $\mu$-a.e. $\omega$.
  
  Then there exists an $S$-invariant probability measure $\eta$ on $\Omega\times X$, with projection to $\eta$ equal to $\mu$, and such that 
  \begin{align*}
   \int_{\Omega\times X} F d\eta \geq l
  \end{align*}
 \end{enumerate}
\end{lemma}

The topic of ``ergodic optimization'' is concerned with results similar to the one above; a more general statement, in which $ F_n(\omega) $ are a subadditive sequence of upper semi-continuous functions, can be found in the paper of Morris \cite[Thm. A.3]{Morris}.

\begin{proof}
 The proof of (ii) is similar to that of (i), so we focus on the latter.
 
 Since $X$ is compact, the continuous function $f_n$ achieves the supremum at some point $x_n$.
 Define a probability measure using Dirac delta-functions on the trajectory of $x_n$ by
 \begin{align*}
  \mu_n := \frac 1 n \left( \delta_{x_n} + \cdots + \delta_{T^{n-1}x_n} \right)
 \end{align*}
 Then by construction $\int f d\mu_n = \frac 1n \sup_x f_n(x)=\frac 1n a_n\geq l$.
 
 Let $\mu$ be a weak-* limit of the $\mu_n$.
 Then $\mu$ is $T$-invariant (by a Krylov--Bogoliubov type argument) and satisfies $\int f d\mu \geq l$ by construction.
 
 For the proof of (ii), one needs to select maximizers of $f_n$ in each fiber of $\Omega\times X \to \Omega$ (see also the end of proof of \autoref{thm:NCET}).
 Namely, one constructs a map $\sigma_n:\Omega \to X$ such that $F_n(\omega)=f_n(\sigma_n(\omega))$.
 Then the measure $\eta_n$ is defined as the average along an $S$-orbit of length $n$ of the measure $(\sigma_n)_*\mu$ and the proof proceeds as before.
\end{proof}

\subsubsection{Measurable selection theorems}
\label{sssec:measurable_selection}
In the above proof, one needs to construct measurable sections of the projection map $ \Omega\times X \to \Omega $ with image in the maximizing set of the functions $ f_n $.
While it is not true that any measurable surjective map $ p:A\to B $ between two Borel spaces has a Borel section, the following variants do hold and either one suffices for the current applications:
\begin{itemize}
	\item[(i)] {\cite[Thm. 6.9.6, Vol. 2]{Bogachev}, \cite[Thm. 35.46]{Kechris}} Let $ X,Y $ be Polish spaces and $ \Gamma\subset X\times Y $ a Borel set such that for all $ x\in X $ the fibers $ \Gamma_x:=\{y\in Y\colon (x,y)\in \Gamma\} $ are $ \sigma $-compact.
	Then $ \Gamma $ contains the graph of a Borel mapping $ f:X\to Y $.
	\item[(ii)] \cite[Thm 3.4.1]{Arveson} Let $ p:P\to \Omega $ be a Borel map from a Polish space $ P $ to a Borel space $ \Omega $.
	Assume that $ p $ maps open sets to Borel sets, and the preimage of any point in $ \Omega $ is a closed set in $ P $.
	Then $ p $ has a Borel section.
\end{itemize}

\subsubsection{Translation distance and horocycles}
For an isometry $g$ of $X$, the distance by which it moves $x_0$ can be measured by
\begin{align}
\label{eqn:dist_via_horofctn}
 d(x_0, gx_0) = \max_{h\in \conj{X}} -h(g^{-1}x_0)
\end{align}
Indeed, assume that $h$ comes from a point $z\in X$.
Then
\begin{align*}
 -h_z(g^{-1}x_0) &= d(z,x_0) - d(z,g^{-1}x_0) \\
 \nonumber & \leq d(x_0,g^{-1}x_0)=d(x_0,gx_0)
\end{align*}
This inequality persists in the closure $\conj{X}$ of $X$.
Taking $z=g^{-1}x_0$ achieves equality.

\subsubsection{A cocycle function.}
Define therefore the function
\begin{align*}
 F: &\Isom(X)\times \conj{X} \to \bR\\
 \nonumber &F(g,h):= - h(g^{-1}x_0)
\end{align*}
If the function $h$ comes from a point $z\in X$, i.e. $h=h_z$, then the following cocycle property holds:
\begin{align*}
 F(g_1 g_2,h_z) &= -\left( d(z,g_2^{-1}g_1^{-1}x_0) - d(z,x_0) \right)\\
 \nonumber &= -\left( d(g_2z,g_1^{-1}x_0) - d(g_2 z, x_0) + d(g_2 z, x_0) - d(z, x_0) \right)\\
 \nonumber &= F(g_1,g_2 h_z) + F(g_2, h_z)
\end{align*}
By continuity of the action of isometries, this property extends to all $h\in \conj{X}$:
\begin{align*}
 F(g_1g_2,h) = F(g_1, g_2 h) + F(g_2,h)
\end{align*}

\subsubsection{The total space of the cocycle.}
Consider the space $\Omega\times \conj{X}\to \Omega$.
Extend the action of $T$ on $\Omega$ to a map $S$ on $\Omega\times \conj{X}$ by $S(\omega,h):=(T\omega,g^{-1}_\omega h)$.
Define now the function
\begin{align*}
 f_1 : \Omega\times \conj{X}& \to \bR\\
 \nonumber f_1(\omega,h)&:= F(g_\omega^{-1},h)
\end{align*}
The Birkhoff sums for $f_1$ and the transformation $S$, using the cocycle property of $F$, can be expressed as:
\begin{align*}
 f_n(\omega,h) & := \sum_{i=0}^{n-1} f_1(S^i (\omega,h))\\
 \nonumber & = F(g_\omega^{-1},h) + F(g^{-1}_{T\omega},g_{\omega}^{-1}h) + \cdots + F(g^{-1}_{T^{n-1}\omega}, g^{-1}_{T^{n-2}\omega}\cdots g^{-1}_\omega h )\\
 \nonumber &= F(g^{-1}_{T^{n-1}\omega}\cdots g^{-1}_\omega ,h)
\end{align*}
By the definition of $F$, this gives
\begin{align*}
 f_n(\omega,h) = -h(g_\omega\cdots g_{T^{n-1}\omega} x_0)
\end{align*}

\subsubsection{Existence of the drift.}
Define the functions
\begin{align*}
 F_n(\omega) := d(g_{T\omega} \cdots g_{T^n\omega} x_0, x_0)
\end{align*}
Then by the triangle inequality $F_{n+m}(\omega) \leq F_n(\omega) + F_m(T^n\omega) $.
Therefore the drift $l$ defined in \eqref{eqn:noncomm_thm_drift} exists by subadditivity of the integrals of $F_n$.
The limit \eqref{eqn:noncomm_thm_local_drift} exists $\mu$-a.e. and equals $l$ by the Kingman ergodic theorem.

\subsubsection{Existence of a maximizing measure.}
Note that by \autoref{eqn:dist_via_horofctn} the functions $F_n$ also satisfy
\begin{align*}
 F_n(\omega) = \max_{x\in\conj{X}} f_n(\omega,x)
\end{align*}
Thus, by \autoref{lemma:maximizing_measure} there exists a measure $\eta$ on $\Omega\times \conj{X}$ such that
\begin{align*}
 \int_{\Omega\times \conj{X}} f_1d\eta \geq l
\end{align*}
Note however than since $F_n(\omega)\geq \sup_{x} f_n(\omega,x)$, it follows that for any $T$-invariant measure the reverse inequality also holds:
\begin{align*}
 \int_{\Omega\times \conj{X}} f_1 d\eta \leq l
\end{align*}

\subsubsection{Existence of $h_\omega$.}
Let now $M$ be the set of all $(\omega,x)$ for which the Birkhoff theorem holds for the measure $\eta$ and the function $f_1$.
Then $M$ is of $\eta$-full measure, so over $\mu$-a.e. $\omega$ the set $M$ is not empty.
Using a measurable selection theorem (see \autoref{sssec:measurable_selection}) there is a map $\omega\mapsto h_\omega\in M$ and by construction it satisfies the required properties. \hfill \qed

\subsection{An example}
\label{ssec:example}

This discussion follows an unpublished note by Karlsson--Monod (see \cite{Karlsson_notes}).
See also the book of Aaronson \cite[Prop. 2.3.1]{Aaronson}.

Consider a function $D:\bR_{\geq 0}\to \bR_{\geq 0}$ which is
\begin{itemize}
 \item Increasing, and $D(0)=0$
 \item $D(t)\to \infty$ as $t\to \infty$
 \item $\frac {D(t)}{t} \to 0$ monotonically as $t\to \infty$
\end{itemize}
This implies that
\begin{align*}
 \frac 1 {t+s} D(t+s) \leq \frac 1t D(t)
\end{align*}
Assuming that $t\leq s$ this gives
\begin{align*}
 D(t+s)\leq D(t) + \frac s t D(t) = D(t) + \frac {D(t)/t}{D(s)/s}D(s)\leq D(t) + D(s)
\end{align*}
Therefore $\bR$ equipped with the distance function $ \dist(x,y):=D(|x-y|) $ is a proper metric space.
One can check that the only point in the metric bordification is in this case the zero function, i.e. $\conj{X}=X\cup \{h\equiv 0\}$.
The isometry group is still $\bR$, acting by translations.
So a function $f:\Omega\to \bR$ can be viewed as a map to the isometry group.

\begin{corollary}
 Assume that $T:(\Omega,\mu)\to (\Omega,\mu)$ is a probability measure-preserving transformation, and $f:\Omega\to \bR$ is $D$-integrable, i.e.
 \begin{align*}
  \int_{\Omega} D(|f|) d\mu < \infty
 \end{align*}
 Then we have
 \begin{align*}
  \lim \frac 1N D\left(\vert f(\omega) + \cdots + f(T^{N-1}\omega) \vert\right) = 0
 \end{align*}
\end{corollary}

\begin{example}[Marcinkiewic--Zygmund]
 Take $D(t):=t^p$ with $0<p<1$.
 Assume that $f\in L^p(\Omega,\mu)$ (note that typically $f\notin L^1(\Omega,\mu)$).
 Then it follows that
 \begin{align*}
  \frac 1 {N} \left( f(\omega) + \cdots + f(T^{N-1}\omega) \right)^p &\to 0
 \intertext{or equivalently} 
  \frac 1 {N^{1/p}} \left( f(\omega) + \cdots + f(T^{N-1}\omega) \right) &\to 0
 \end{align*}	
\end{example}

\section{Mean versions of the Multiplicative Ergodic Theorem}
\label{sec:mean_MET}

This section contains an $L^2$-analogue of the Multiplicative Ergodic Theorem.
The geometry of non-positively curved spaces allows a streamlined presentation of the results.
The setting is that of $\CAT(0)$-spaces, but the results hold in the more general case of metric spaces which are uniformly convex and non-positively curved in the sense of Busemann (see e.g. \cite{Karlsson_Margulis}).

The definitions and basic properties of $\CAT(0)$-spaces are discussed in \autoref{ssec:cat0_spaces}; a more comprehensive treatment can be found in the monograph of Bridson--Haefliger \cite{Bridson_Haefliger}.
Direct integrals of metric spaces are defined in \autoref{ssec:direct_integrals}, following Monod \cite{Monod}.
These are ``big'' metric spaces which inherit nice curvature properties of their constituents.
Using the developed formalism, a Mean Multiplicative Ergodic Theorem is proved in \autoref{ssec:mMET}.
We end in \autoref{ssec:mean_kingman} with a mean form of Kingman's subadditive ergodic theorem.

Throughout this section, $X$ will be a metric space equipped with a distance $\dist(-,-)$.
The space will be assumed complete, but not necessarily proper (i.e. locally compact).

\subsection{\texorpdfstring{$\CAT(0)$}{CAT(0)} spaces}
\label{ssec:cat0_spaces}

Below we recall two more equivalent definitions of $\CAT(0)$ spaces.
Depending on the context, one is easier to use than the other.

\subsubsection{Definition via comparison triangles}
\label{def:cat0_comparison}
One definition already appears in \autoref{sssec:cat0} and says that triangles are thinner than their comparison analogues in Euclidean space.
A variant (which can be easier to check) is as follows.
Let $ [y,z] $ be a geodesic segment and $ m\in [y,z] $ its midpoint.
For any other point $ x\in X $ there exist points in Euclidean space $ x',y',z' $ such that the pairwise distances between $ x,y,z $ and $ x',y',z' $ coincide; furthermore let $ m'\in [y',z'] $ be the corresponding midpoint.
Then the (complete) metric space $ X $ is $ \CAT(0) $ if $ \dist (x,m)\leq \dist (x',m') $.

\subsubsection{Definition via the parallelogram law}
\label{def:cat0_parallelogram}
The space $X$ is $\CAT(0)$ if for any $y_1,y_2\in X$ there exists a point $m\in X$ such that the following holds $\forall x\in X$:
\begin{align*}
 \dist(x,m)^2 \leq \frac 12 \left( \dist(x,y_1)^2 + \dist(x,y_2)^2 \right) - \frac 14 \dist(y_1,y_2)^2
\end{align*}
By appropriate choices of $x$, it follows that $m$ is the \emph{unique} midpoint of $[y_1,y_2]$.
Note that if in a complete metric space any two points have a unique midpoint, then the space is geodesic, and moreover geodesics are unique.

\subsubsection{Convexity}
\label{sssec:convexity}
The geometry of a $\CAT(0)$ space enjoys a number of convexity properties, some of which are summarized below.
A set $C\subset X$ is \emph{convex} if it contains the geodesic between any two of its points.

\begin{enumerate}
 \item For any point $x$ and geodesic $\gamma$ in $X$, the distance function $\dist(x,\gamma(t))$ is convex.
 This follows after some algebraic manipulations from \autoref{def:cat0_parallelogram}.
 \item For two geodesic rays $\gamma_1,\gamma_2$ the function $\dist(\gamma_1(a_1 t),\gamma_2(a_2 t))$ is convex in $t$, for any $ a_1,a_2 \geq 0 $.
 \item For two geodesic rays  $\gamma_1,\gamma_2$ that start at the same point, the function $\frac 1t\dist(\gamma_1(t),\gamma_2(t))$ is semi-increasing.
 \item The distance function to a convex set (defined on all of $X$) is itself a convex function.
 \item For a convex subset $C$ the nearest point projection map $p:X\to C$ is well-defined and distance-decreasing.
\end{enumerate}

\subsubsection{Uniform convexity}
A metric space $X$ admitting midpoints is \emph{uniformly convex} if there exists a strictly decreasing continuous function $g$ on $[0,1]$ with $g(0)=1$ satisfying:
For all $x,y_1,y_2\in X$, with $m$ the midpoint of $[y_1,y_2]$ and $R:=\max(\dist(x,y_1),\dist(x,y_2))$ we have
\begin{align*}
 \frac{\dist(x,m)}{R} \leq g\left(\frac{\dist(y_1,y_2)}{R} \right)
\end{align*}
The function $g$ is called the modulus of convexity.

For $\CAT(0)$-spaces, the function $g(\epsilon):=(1-\frac 14\epsilon^2)^{\frac 12}$ works.
This follows directly from the definition in \autoref{def:cat0_parallelogram}.
We will need the following result, reproduced from \cite[Lemma 3.1]{Karlsson_Margulis}, which is valid for all uniformly convex metric spaces.

\begin{lemma}
\label{lemma:reverse_triangle}
 Let $x,y,z\in X$ be points satisfying an almost reverse triangle inequality:
 \begin{align*}
  (1-\epsilon) \dist(x,y) + \dist(y,z) \leq \dist(x,z)
 \end{align*}
 for some $\epsilon>0$.
 Let $y'$ be the point on the geodesic $[x,z]$ such that $\dist(x,y')=\dist(x,y)$.
 Then there exists $f(\epsilon)>0$ with $f(\epsilon)\to 0$ as $\epsilon \to 0$ such that
 \begin{align*}
  \dist(y,y') \leq f(\epsilon) \dist(x,y)
 \end{align*}
\end{lemma}
\begin{proof}
 From the definition of $y'$ and the assumption of the lemma, it follows that
 \begin{align*}
  \max\left( \dist(z,y), \dist(z,y') \right) \leq \dist(x,z) - (1-\epsilon)\dist(x,y)
 \end{align*}
 Setting $m$ to be the midpoint of $[y,y']$, the uniform convexity of $X$ gives
 \begin{align*}
  \dist(z,m) \leq \max\left( \dist(z,y), \dist(z,y') \right)
 \end{align*}
 The triangle inequality gives
 \begin{align*}
  \dist(x,m) &\geq \dist(x,z) - \dist(z,m)
  \intertext{which combined with the last two inequalities gives}
  \dist(x,m) &\geq (1-\epsilon) \dist(x,y)
 \end{align*}
 Using this inequality and setting $R=\dist(x,y)=\dist(x,y')$, the definition of uniform convexity gives
 \begin{align*}
  1-\epsilon \leq \frac{\dist(x,m)}{R} \leq g\left( \frac{\dist(y,y')}{2R} \right)
 \end{align*}
 Since $g$ is continuous and strictly decreasing with $g(0)=1$, it follows that $\frac{\dist(y,y')}{2R}\leq f(\epsilon)$ with $f(\epsilon)\to 0$ as $\epsilon \to 0$.
 This is the required conclusion.
\end{proof}

\subsubsection{Boundaries}
\label{sssec:boundaries}
Starting from a $\CAT(0)$-space $X$, one can associate in a natural way a boundary $\partial X$ to it.
When $X$ is locally compact, the union equipped with a natural topology $\conj{X}=X\coprod \partial X$ is compact.
Two constructions are possible: one using equivalence classes of geodesics, and another using horofunctions.
For $\CAT(0)$-spaces, these give naturally isomorphic boundaries.

\subsubsection{Geodesic, or visual, boundary}
Define $\partial_{geod} X$ to consist of equivalence classes of geodesic rays $\gamma:[0,\infty)\to X$, where two geodesics are equivalent if they are a bounded distance away from each other.
Alternatively, fix a basepoint $x_0\in X$ and define the boundary to equal all geodesic rays starting at $x_0$.
Each geodesic ray is equivalent to a unique one starting at $x_0$, so the boundary is independent of the choice of basepoint.

\subsubsection{Horofunction boundary}
The space $X$ naturally embeds in the space of $1$-Lipschitz function, and the horofunction, or metric bordification of $X$ is described in \autoref{def:metric_bord}.
The equivalence between the geodesic and horofunction boundaries is given by Busemann functions, \autoref{def:buseman_fctn}.

\subsection{Direct integrals of \texorpdfstring{$\CAT(0)$}{CAT(0)} spaces}
\label{ssec:direct_integrals}

The presentation in this section follows the article of Monod \cite{Monod}.
The concept of direct integral of $\CAT(0)$-spaces is developed there and further used to prove superrigidity-type statements for product group actions.

A slightly more general setting would be that of uniformly convex, non-positively curved metric spaces in the sense of Busemann.
Then, one could take $L^p$-norms (with $1<p<\infty$) in \autoref{def:direct_integral} of direct integrals of metric spaces.

\subsubsection{Setup}
Let $(\Omega,\mu)$ be a Borel measure space, and let $\crX \xrightarrow{p} \Omega$ be a bundle of $\CAT(0)$-spaces.
For instance, the bundle could be trivial: $\crX=X\times \Omega$ where $X$ is a $\CAT(0)$-space.
A \emph{basepoint} for the bundle is a section $\sigma_0:\Omega\to \crX$ such that $p\circ \sigma_0 = \id_\Omega$.
Denote by $L^2(\Omega,\mu)$ the space of real valued functions on $\Omega$ which have finite $L^2$-norm for the measure $\mu$ (up to $\mu$-a.e. equivalence).

\begin{definition}
\label{def:direct_integral}
 The \emph{direct integral} of metric spaces is the set of measurable sections whose distance to $\sigma_0$ is in $L^2(\Omega,\mu)$:
 \begin{align*}
  L^2(\crX,\mu):=\lbrace \sigma:\Omega\to \crX \mid \dist(\sigma_0,\sigma)\in L^2(\Omega,\mu) \rbrace.
 \end{align*}
 and two sections are identified if they agree $\mu$-a.e.
 The distance between two sections is defined by
 \begin{align}
 \label{def:l2_section_dist}
  \dist(\sigma_1,\sigma_2)^2 := \int_\Omega \dist(\sigma_0(\omega), \sigma(\omega))^2 \, d\mu(\omega) <\infty
 \end{align}
 The integral defining the distance is finite by pointwise comparison: two sections at finite $L^2$-distance from $\sigma_0$ are at finite $L^2$-distance from each other.
\end{definition}

\subsubsection{Properties}
Checking the definitions one finds that:
\begin{enumerate}
 \item The metric space $L^2(\crX,\mu)$ is complete and separable, if $\mu$-a.e. fiber of $\crX$ is.
 \item The direct integral of $\CAT(0)$-spaces is itself $\CAT(0)$.
 This is immediate from the definition via the parallelogram law \autoref{def:cat0_parallelogram}.
 \item The midpoint $\sigma_m$ of two sections $\sigma_x,\sigma_y$ is given by the pointwise midpoint in each fiber.
\end{enumerate}

\subsubsection{Geodesics}
\label{sssec:geodesics}
A geodesic $\gamma:[0,r]\to L^2(\crX,\mu)$ is equivalent to the data of
\begin{enumerate}
 \item A function $\alpha\in L^2(\Omega,\mu)$ satisfying $\int_\Omega \alpha^2 \, d\mu=1$ (called a \emph{semi-density}).
 \label{def:semidensity}
 \item A (measurable) collection of geodesics $\gamma_\omega:[0,r\cdot \alpha(\omega)] \to \crX_\omega$ satisfying
 \begin{align*}
 \gamma(t)(\omega) = \gamma_\omega( t \cdot \alpha(\omega))
 \end{align*}
\end{enumerate}
A verification of this description is available in \cite[Prop. 44]{Monod}.

\subsubsection{Boundaries}
The boundary of $L^2(\crX,\mu)$ is a join integral:
\begin{align*}
 \partial L^2(\crX,\mu) := \int_\Omega^* \left(\partial \crX_\omega\right) \, d\mu(\omega)
\end{align*}
where an element of the right-hand side is the data $(\phi,\alpha)$ of a measurable section $\phi(\omega)\in \partial \crX_\omega$ and a semi-density $\alpha$ (see \autoref{sssec:geodesics}\ref{def:semidensity}).
This description is detailed in \cite[Rmk. 48]{Monod}.

\subsubsection{Induced actions}
\label{sssec:induced_actions}
Suppose now that the bundle $\crX\to \Omega$ is equipped with an action $T$ by fiberwise isometries.
This gives a transformation $T\curvearrowright\Omega$ and isometries $T_\omega : \crX_\omega \to \crX_{T\omega}$.

The action of $T$ then naturally extends to sections of the bundle by the formula
\begin{align}
\label{eqn:section_pull_back}
 (T^*\sigma)(\omega) := T_\omega^{-1}\left( \sigma(T\omega) \right).
\end{align}
Assume that the action satisfies the $L^2$-integrability condition $\dist(\sigma_0, T^*\sigma_0)<\infty$.
Then we have an action by pullback on the entire space $L^2(\crX,\mu)$.

\begin{remark}
 Inspecting the definition of pullback of sections in \autoref{eqn:section_pull_back} it is more natural to define the induced action on the metric space bundle $\crX$ by
 \begin{align*}
  T_\omega: \crX_{T\omega} \to \crX_\omega
 \end{align*}
 Indeed, this would be more compatible with the previous sections, where $\crX$ is the bundle of fiberwise metrics on a vector bundle.
 Given a linear map between vector spaces, the induced map on the space of metrics is by pullback, and goes the other way.
\end{remark}

\subsection{Mean Multiplicative Ergodic Theorem}
\label{ssec:mMET}

The construction of direct integrals of $\CAT(0)$ spaces described in the previous section allows for a rather direct proof of a mean form of the Multiplicative Ergodic Theorem.
The discussion below follows Karlsson \& Margulis \cite{Karlsson_Margulis} and illustrates their proof in the setting of a single semi-contraction of a $\CAT(0)$-space.

\begin{theorem}
\label{thm:cat0_isometry}
 Let $X$ be a complete $\CAT(0)$ space and $T:X\to X$ be a semi-contraction (i.e. $\dist(Tx,Ty)\leq \dist(x,y)\quad \forall x,y\in X$).
 
 Then there exists a number $A\geq 0$, and a geodesic ray $\gamma$, such that for any starting point $x_0\in X$ we have that
 \begin{align*}
  \frac 1 n \dist\left( T^n x_0, \gamma(A\cdot n) \right) \xrightarrow{n\to +\infty} 0
 \end{align*}
\end{theorem}

\autoref{thm:cat0_isometry} applies immediately to the direct integral of $\CAT(0)$-spaces to give the following mean version of the Multiplicative Ergodic Theorem.
\begin{corollary}
\label{cor:mean_non_commutative}
 Let $T\curvearrowright (\Omega,\mu)$ be an ergodic probability measure preserving system, and let $\crX\to \Omega$ be a bundle of $\CAT(0)$-spaces.
 Assume given a section $\sigma_0:\Omega\to\crX$ and an induced action of $T$ on $\crX$ by fiberwise isometries, such that the integrability assumptions from \autoref{sssec:induced_actions} are satisfied.
 
 Then there exists $A\geq 0$ and geodesic rays $\gamma_\omega \in \crX_\omega$ such that the orbits of $\sigma_0$ track sublinearly the geodesics $\gamma_\omega$.
 More precisely:
 \begin{align*}
  \frac 1 n \left( \int_{\Omega} \dist \big( T^n_\omega \sigma_0(\omega), \gamma_{T^n\omega}(A\cdot n) \big)^2 \, d\mu(\omega) \right)^{\frac 12} \xrightarrow{n\to +\infty} 0
 \end{align*}
\end{corollary}

\begin{proof}[Proof of \autoref{cor:mean_non_commutative}]
 The result follows from \autoref{thm:cat0_isometry} by considering the induced action, in the sense of \autoref{sssec:induced_actions}, on the $\CAT(0)$-space $L^2(\crX,\mu)$.
\end{proof}

\begin{proof}[Proof of \autoref{thm:cat0_isometry}]
Define $x_n := T^n x_0$.
By the discussion in \autoref{sssec:linear_drift} and \autoref{sssec:semi-contractions}, the subadditive sequence $a_n := \dist(x_0,x_n)$ has linear growth:
\begin{align}
\label{eqn:lim_drift}
 \lim_N \frac 1N a_N = A \geq 0.
\end{align}
Assume that $A>0$, otherwise the theorem is immediate.

We will construct successive approximations to the tracking geodesic, and all our geodesic will start at the basepoint $x_0$.
Fix a sequence $\epsilon_i>0$ whose speed of decrease to zero will be described later.

\subsubsection{Picking the record holders}
For now freeze $\epsilon_i>0$.
Because of \autoref{eqn:lim_drift}, there exists $K_i$ such that
\begin{align}
\label{eqn:a_n_control}
 (A-\epsilon_i)\cdot n \leq a_n \leq (A+\epsilon_i)\cdot n \quad \forall K_i \leq n
\end{align}

The sequence $b_{n}= a_n - (A-\epsilon_i)\cdot n$ diverges to $+\infty$ as $n\to \infty$.
Thus there exists arbitrarily large $N_i$ such that $b_{N_i}$ is larger than all of the previous elements, i.e. we have
\begin{align}
 a_{N_i} - (A-\epsilon_i)\cdot N_i & \geq a_n - (A - \epsilon_i)\cdot n && \forall n\leq N_i \nonumber
 \intertext{or rewriting it:}
 a_{N_i} - a_n &\geq (A-\epsilon_i)\cdot (N_i-n) && \forall n\leq N_i
 \label{eqn:record}
\end{align}
Pick $N_i$ such that $N_i > K_{i+1}$ where $K_i$ is defined by \autoref{eqn:a_n_control}.

\subsubsection{Geometry of record holders}
\label{sssec:geometry_record}
We now rewrite the above inequalities in terms of distances, valid for all $x_n$ with $K_i\leq n \leq N_i$.
From \autoref{eqn:record} we find
\begin{align*}
 \dist(x_0, x_{N_i}) & \geq \dist(x_0, x_n) + (A-\epsilon_i)\cdot (N_i -n)\\
		     & \geq \dist(x_0, x_n) + \frac{A-\epsilon_i}{A+\epsilon_i} \cdot \dist(x_0,x_{N_i - n}) && \textrm{by \autoref{eqn:a_n_control}}\\
		     & \geq \dist(x_0, x_n) + \frac{A-\epsilon_i}{A+\epsilon_i} \cdot \dist(x_n,x_{N_i}) && \textrm{by semi-contraction.}
\end{align*}
The inequality just obtained is exactly the almost reverse triangle inequality to which \autoref{lemma:reverse_triangle} will apply.

\subsubsection{Constructing the geodesic}
Define $\gamma_i$ to be the geodesic connecting $x_0$ and $x_{N_i}$; it has length $a_{N_i}$.
Let also $x_{N_{i-1}}'$ be the point on the geodesic $\gamma_{i}$ at distance $a_{N_{i-1}}$ from $x_0$.
Recall that by construction we have $K_{i}<N_{i-1}<N_{i}$.
Therefore the inequality in \autoref{sssec:geometry_record} holds for $n=N_{i-1}$ and so \autoref{lemma:reverse_triangle} gives
\begin{align}
\label{eqn:fixed_time_bound}
 \dist(x_{N_{i-1}},x_{N_{i-1}}') \leq f\left( \frac{2\epsilon_i}{A+\epsilon_i} \right) \dist(x_0, x_{N_{i-1}})
\end{align}
Now for any $r\leq \dist(x_0,x_{N_{i-1}})$, we have the points $\gamma_i(r)$ and $\gamma_{i-1}(r)$.
For $r = \dist(x_0,x_{N_{i-1}})$ \autoref{eqn:fixed_time_bound} gives a bound for the distance between points of the form
\begin{align*}
 \dist(\gamma_{i-1}(r),\gamma_i(r))\leq f\left( \frac{2\epsilon_i}{A+\epsilon_i} \right) \cdot r.
\end{align*}
By the convexity property of distances in \autoref{sssec:convexity}(iii) it follows that the same bound holds for all $r\leq \dist(x_0,x_{N_{i-1}})$.

Pick now the sequence $\epsilon_i\to 0$ such that $f\left( \frac{2\epsilon_i}{A+\epsilon_i} \right)\leq 2^{-i}$; this is possible since $f(\epsilon)\to 0$ as $\epsilon\to 0$.
For fixed $r$ the sequence of points $\gamma_i(r)$ form a Cauchy sequence and converge to a point $ \gamma(r) $ and giving a geodesic ray $\gamma$ starting at $x_0$.

\subsubsection{Tracking property}
It remains to check that the orbit $\{x_n\}$ stays near the geodesic $\gamma$.
For any $k$ there is an $i$ such that $K_i \leq k \leq N_i$.
The triangle inequality and the previous bounds give:
\begin{align*}
 \dist(\gamma(A\cdot k), x_k) & \leq \dist(\gamma(Ak), \gamma_i(A k)) + \dist(\gamma_i(A k), \gamma_i(a_k)) + \dist(\gamma_i(a_k), x_k)\\
 & \leq 2^{-i+1}\cdot A \cdot k + |A\cdot k - a_k| + f\left( \frac{2\epsilon_i}{A+\epsilon_i} \right)\cdot a_k\\
 & \leq 2^{-i+1}\cdot A \cdot k + |A\cdot k - a_k| + 2^{-i} \cdot a_k
\end{align*}
It follows then immediately that $\frac 1k \dist(\gamma(A\cdot k), x_k) \to 0$, which ends the proof.
\end{proof}

\begin{remark}\leavevmode
\begin{enumerate}
 \item The above proof works in the general setting of uniformly convex, non-positively curved in the sense of Busemann metric spaces (see \cite{Karlsson_Margulis}).
 \item The geodesic in the conclusion of \autoref{thm:cat0_isometry} is unique, assuming the basepoint is fixed.
 This follows from the convexity property in \autoref{sssec:convexity}(iii).
 \end{enumerate}
\end{remark}

\subsection{Mean Kingman theorem}
\label{ssec:mean_kingman}

We end with an $L^2$-version of Kingman's Subadditive Ergodic Theorem, under a slightly stronger assumption.

\begin{proposition}
 Let $T \curvearrowright (\Omega,\mu)$ by an ergodic probability measure-preserving action, and let $U\curvearrowright L^2(\Omega,\mu)$ be the induced unitary operator on functions.
 Let $f_n\in L^2(\Omega,\mu)$ be a sequence of functions satisfying $f_n\geq 0$ pointwise $\mu$-a.e. and
 \begin{align*}
  f_{n+m} \leq f_n + U^n f_m \quad \textrm{pointwise $\mu$-a.e.}
 \end{align*}
 Denote the averages by $a_n:= \ip{\id_{\Omega}, f_n}$, which is a subadditive sequence with limit $\frac 1n a_n \to A$.
 Then we have
 \begin{align*}
  \lim_{n\to +\infty} \norm{\frac 1 n f_n - A}_{L^2} = 0.
 \end{align*}
\end{proposition}

The proof below will use von Neumann's ergodic theorem.

\begin{proof}
 Note that $\ip{f_n,A\cdot \id_\Omega}=a_n\cdot A$, so expanding $\norm{\frac 1n f_n - A}^2$ it suffices to show that
 \begin{align*}
  \norm{\frac 1n f_n}^2 - A^2 \to 0.
 \end{align*}
 Using again $\ip{f_n,\id_\Omega}=a_n$ and Cauchy--Schwarz gives
 \begin{align*}
  \norm{\frac 1n f_n}^2 \geq \left(\frac{a_n}{n}\right)^2 \to A^2
 \end{align*}
 so it suffices to show that $\limsup \norm{\frac 1n f_n}^2 \leq A^2$.
 Fix now $k>0$ and divide $N$ with remainder as $N= k\cdot l + r$ with $0\leq r <k$.
 Applying the subadditivity assumption iteratively gives the pointwise bound
 \begin{align}
  f_N &\leq f_k + U^k f_k + \cdots + U^{k\cdot (l-1)} f_k + U^{k\cdot l} f_r \nonumber \\
  \intertext{and after dividing out by $N$ and applying straightforward estimates}
  \frac 1N f_N & \leq \frac 1l \left( \left(\frac 1k f_k\right) + \cdots + U^{k\cdot (l-1)} \left(\frac 1k f_k\right) \right) + O\left(\frac 1N\right).
  \label{eqn:vN_estimate}
 \end{align}
 Recall from the von Neumann Mean Ergodic Theorem applied to $U^k$ that for any $g\in L^2(\Omega,\mu)$ we have
 \begin{align*}
  \frac 1 l \left( g + U^k g + \cdots + U^{k\cdot (l-1)}g \right) \xrightarrow{L^2} P_k(g)
 \end{align*}
 where $P_k$ denotes orthogonal projection onto $U^k$-invariant vectors.
 Note also that the space of $U^k$-invariant vectors coincides with the direct sum of eigenspaces for $U$ with eigenvalues $\zeta_d$, where $\zeta_d^d=1$ are roots of unity and $d$ divides $k$.
 
 Pick now two arbitrarily large $k_1,k_2$ which are relatively prime, and apply the estimate from \autoref{eqn:vN_estimate} with $k=k_1,k_2$.
 From the assumed positivity of the functions, it follows that
 \begin{align*}
  \limsup \norm{\frac 1N f_N}^2 \leq \ip{P_{k_1}\left(\frac 1{k_1} f_{k_1}\right), P_{k_2}\left(\frac 1{k_2} f_{k_2}\right)  }
 \end{align*}
 Finally, note that $P_{k_i}(f_{k_i}) = a_{k_i} \cdot \id_\omega + h_i$ where $h_i$ is in the $U$-eigenspaces corresponding to non-trivial roots of unity dividing $k_i$.
 In particular since $k_1$ and $k_2$ are relative prime, it follows that the $h_i$ and $\id_\Omega$ are orthogonal.
 
 Since $\frac 1 {k_i} a_{k_i}$ tends to $A$ as $k_i\to \infty$, the desired upper bound follows.
\end{proof}

\paragraph{Acknowledgements}
These notes are based on lectures given at several summer schools.
I am very grateful to the organizers of the summer schools in Ilhabela, Brazil (January 2015, ``Holomorphic Dynamics School''), Luminy, France (July 2015, ``Translation Surfaces School''), and  Moscow, Russia (May 2016 ``Hyperbolic Geometry and Dynamics School'').

These notes were strongly influenced and owe a great debt to the point of view in the works of Kaimanovich, Karlsson, Ledrappier, Margulis, Monod, Zimmer.
I am also grateful for remarks and suggestions from Alex Eskin, Pascal Hubert, Misha Kapovich, Zemer Kosloff, Erwan Lanneau, Carlos Matheus, Martin M\"{o}ller, Nessim Sibony, Misha Verbitsky, Anton Zorich.
I am also grateful to the referee for a careful reading and many suggestions and comments that significantly improved the text.

This research was partially conducted during the period the author served as a Clay Research Fellow.



\appendix


\bibliographystyle{sfilip}
\bibliography{MET_lectures}
\end{document}